\documentclass[reqno,12pt]{article}

\usepackage[utf8]{inputenc}
\usepackage[scaled=0.875]{helvet} 
\usepackage{courier}
\usepackage{graphicx}
\usepackage{color} 
\usepackage{authblk}
\usepackage{geometry}
\usepackage[english]{babel}
\usepackage{amsmath,amsfonts,amssymb,amsthm}
\definecolor{steelblue}{RGB}{70, 130, 180}
\usepackage[colorlinks=true,linkcolor=steelblue,citecolor=blue]{hyperref} 

\usepackage{mathrsfs} 
\usepackage{tikz}
\usetikzlibrary{arrows,decorations.pathmorphing,backgrounds,positioning,fit,petri}
\usetikzlibrary{shapes}
\usetikzlibrary{decorations.shapes}
\usepackage[babel=true,kerning=true]{microtype} 

\tikzset{paint/.style={ draw=#1!50!black, fill=#1!50 },
    decorate with/.style=
    {decorate,decoration={shape backgrounds,shape=#1,shape size=2.5pt}}}

\newcommand{\N}{{\mathbb{N}}}

\newcommand{\Z}{{\mathbb{Z}}}

\newcommand{\PP}{{\mathbb{P}}}
\newcommand{\1}{{\mathbf{1}}}
\let\mathcal\mathscr 
\let\geq\geqslant  
\let\leq\leqslant 
\def\AA{\mathcal{A}}

\def\FF{\mathcal{F}}
\def\MM{\mathcal{M}}
\def\NN{\mathcal{N}}

\def\LL{\mathcal{L}}

\def\UU{\mathcal{U}}
\def\RR{\mathcal{R}}
\def\II{\mathcal{I}}
\def\SS{\mathcal{S}}
\def\pp{\mathcal{P}}
\def\YY{\mathcal{Y}}
\def\RR{\mathcal{R}}
\def\ee{\mathrm{e}}

\def\blambda{\overline \lambda}

\def\aa{\mathtt{a}}
\def\cc{\mathtt{c}}

\def\gg{\mathtt{g}}

\renewcommand\tt{\mathtt{t}}

\theoremstyle{plain}
\newtheorem{theo}{Theorem}[section]
\newtheorem{coro}[theo]{Corollary}

\newtheorem{prop}[theo]{Proposition}
\newtheorem{lemm}[theo]{Lemma}

\newtheorem{rema}[theo]{Remark}

\newtheorem*{exam*}{Example}

\makeatletter
\@addtoreset{equation}{section}

\date{\vspace{-5ex}}
\title{Ergodicity of some dynamics of DNA sequences}

\begin{document}

\author[1]{Mikael {\sc Falconnet}}
\author[2]{Nina {\sc Gantert}}
\author[3]{Ellen {\sc Saada}}
\affil[1]{LPO Gustave Eiffel, Bordeaux, France
}
\affil[2]{Technische Universit\"at M\"unchen, 
Fakult\"at f\"ur Mathematik, Germany
}
\affil[3]{CNRS, UMR 8145 laboratoire MAP5, Universit\'e de Paris -- Campus Saint-Germain-des-Pr\'es, France
}

\renewcommand\Authands{ and }

\maketitle  

\begin{center}
{\bf To Guy Fayolle, on the occasion of his birthday}
\end{center}

\begin{abstract} 
We   define interacting particle systems on configurations of  the integer lattice 
(with values in some finite alphabet) by the 
superimposition of two dynamics: 
a substitution process 
 with finite range rates, 
and a circular permutation mechanism (called  ``cut-and-paste'')  
with  possibly  unbounded  range. 
The model is motivated by the dynamics of DNA sequences: we consider 
an ergodic model for substitutions, the RN+YpR model (\cite{piau:solv}),
with three particular cases, 
the  models JC+$\cc$p$\gg$, T92+$\cc$p$\gg$, and RNc+YpR. 
We investigate whether they remain ergodic 
with the additional cut-and-paste mechanism, which models insertions and deletions of nucleotides. 
Using either duality or attractiveness techniques, 
we provide various sets of sufficient conditions, concerning only the substitution rates, 
 for ergodicity of  the superimposed process.
They  imply ergodicity of the models JC+$\cc$p$\gg$, T92+$\cc$p$\gg$  
 as well as the attractive RNc+YpR, all with an additional
 cut-and-paste mechanism.
 \end{abstract}

\section{Introduction}\label{sec:intro}
Motivated by biological models for the evolution of DNA sequences,  
this paper contributes to the classical topic of ergodicity of 
interacting particle systems (see e.g. \cite{ligg:IPS}). 
Let us first give the biological context, then explain the interacting 
particle system it induces, and how we tackle its ergodicity.
This last point corresponds to the intriguing question of the ergodicity
for a superimposition of an ergodic particle system (here a  generalized  spin system)  
and of a non ergodic one (here cyclic permutations which generalize exclusion processes). 
The spin values belong to a finite alphabet of size larger than two, 
which adds a difficulty. \\ 

\noindent \textbf{The biological set-up.}
The study of evolutionary relationships among living organisms has
entered the genomic age in the past decades. 
To model the evolution of DNA sequences remains an important
and difficult part of phylogenetic analysis.
Generally, every method of
phylogenetic reconstruction at the molecular level is based on a
probabilistic  model, for codons  or for  nucleotides. 
Jukes \& Cantor, in \cite{jukes:69},
were the first to introduce a probabilistic model, the (JC) model, to study the changes
in DNA sequences. In the (JC) model, DNA sequences are encoded
as elements of $\AA^N$, where  the positive integer $N$ stands for the
number of nucleotides in one strand of the DNA molecule, and $\AA$ for
the  nucleotide  alphabet $\{\aa,\tt,\cc,\gg  \}$,  where the  letters
represent  adenine, thymine,  cytosine and  guanine  respectively. The
(JC)  model   deals  with  the  product  of  $N$
independent  identically  distributed  continuous-time  Markov  chains
modelling single site nucleotide substitutions.  

The substitution-rate matrix in the (JC)  model is the simplest
possible because all substitutions occur at the same
rate. Since then, other nucleotide substitution processes have been 
introduced to refine this matrix (see for 
instance \cite{kimu:80,fels:81,hase:85,tamu:92,tamu:93,yang:94b})     until
the  generalized time  reversible (GTR) model introduced in~\cite{tava:GTR}
which is such  that the Markov process is reversible,  and with no more
restriction on  the structure of  the matrix. In all  these processes,
the independence assumption on sites was kept.  As a  
consequence, there exists  a unique
stationary probability measure for  the process, which is product.  
This  means  for  instance   that  in  a  long DNA  sequence  at
equilibrium  the frequency  of a  dinucleotide $x$p$y$  should  be the
product of the $x$ and $y$ frequencies (where, for subsets $X$ and $Z$ 
of  $\AA$, $X$p$Z$ is the collection of
dinucleotides in $X\times Z$, and we write $x$p$y$ instead of $\{x\}$p$\{y\}$). 

But  this is  actually  not  the case  in  some biological  contexts.
Indeed,  since the studies  of \cite{Josse}  and \cite{Swartz},  it is
well known that the dinucleotide $\cc$p$\gg$ is less frequently present
in   many  mammals   DNA  than   it  would   be  expected   from  base
composition.  Support for  the  $\cc$p$\gg$
deficiency to be related to  DNA methylation was provided in \cite{bird:CpG}: 
the substitution rate of
cytosine by thymine is higher in methylated $\cc$p$\gg$ than in 
other  dinucleotides.  Therefore,  more realistic  substitution models
incorporating  such  neighboring   effects  have  been  introduced  by
\cite{duret:CpG}  with their Tamura+$\cc$p$\gg$  model.  To  evade the
dependency   between  neighbors,  in \cite{duret:CpG} there is an approximation for
frequencies  of trinucleotides to  capture some  features of  the true
model. B\'erard et al. in \cite{piau:solv} extended the latter model 
to   the  RN+YpR  model   and  assessed   rigorously  the   effect  of
neighbor-dependent substitutions. There,  DNA sequences are encoded as
(doubly infinite) elements of  $\AA^\Z$ and their dynamics are studied
through the techniques of interacting particle systems. The properties
of the RN+YpR model have been used to infer phylogenetic distances in
\cite{falc:dist} or $\cc$p$\gg$ hypermutability rates in \cite{BerardGueguen}.  

But substitutions  are not  the only way  to alter DNA  sequences. For
example, one  may add several extra  nucleotides to a  DNA sequence by
insertions, or remove  them by deletions. In the  model of \cite{TKF},
single sites are  inserted or deleted with rates  independent of their
positions, and this is superimposed to independent
substitutions.  There,  DNA sequences have a variable (but finite)
length  along   time  and  are  encoded  as   elements of $\bigcup_N
\AA^N$. We would like to do for  neighbor-dependent substitution processes
on DNA sequences what  was done in \cite{TKF} for independent substitutions. 
The DNA sequences are viewed, as above, as elements  of $\AA^\Z$.
To avoid
mathematical difficulties due to the insertion-deletion mechanisms (an
insertion of a single nucleotide induces a shift of the whole sequence
and  leads  to  infinite  range  interactions), we require that an
insertion and a deletion occur at  the same time. To that purpose, we
introduce a mechanism that we call  ``cut-and-paste'', whose name is
inspired from the classification  proposed by  \cite{Finnegan}  for
transposable  elements. The  latter, discovered  by \cite{McClintock},
are a type  of DNA that can  move around within the genome  and can be
distinguished in two classes : class I is  commonly called
``copy-and-paste'', and  class II, ``cut-and-paste''. In our settings, we
consider the simplest cut-and-paste mechanism, that is, the transfer
of  one   nucleotide  into  the  sequence   as shown  on
Figure~\ref{figu:InsDel}. There, the transfer of nucleotide $\tt$ can be seen
as its deletion and reinsertion three nucleotides further. 
\begin{figure}[ht]
  \includegraphics{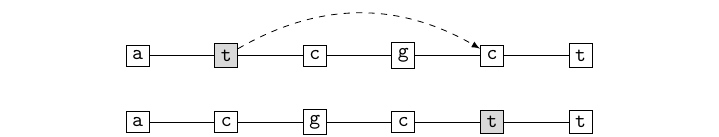}
\caption{One example of the cut-and-paste of a nucleotide into the sequence}
\label{figu:InsDel}
\end{figure} \\ 
\noindent \textbf{The modelling by a particle system, and the question of its ergodicity.}
We consider configurations that  can take values 
in some finite alphabet on each site of the integer lattice, as detailed above; 
 thus  the alphabet will consist of the nucleotides $\aa, \cc, \gg$ 
and $\tt$, but the questions we ask are not relying on this specific choice.
On this configuration space, we superimpose two dynamics.
The first one is given by a substitution process, and the second 
one by a cut-and-paste mechanism with a rate $\rho>0$. 
The substitution process is quite general and includes for instance 
the stochastic Ising model with finite rate of dependence. The cut-and-paste 
process can have unbounded range and permutes the values of a certain 
interval of sites; 
it can be seen as a generalization of an exclusion process. 

In this superimposition, the first dynamics is ergodic but not the second one. 
 Note that ergodicity of our process, at least in the case of finite range permutations, would follow from the ``Positive Rates Conjecture'' which says that in dimension $1$, a finite range interacting particle system with strictly positive transition rates is ergodic, see \cite{ligg:IPS}, page 201. The status of this conjecture is not clear to us, see
 \cite{Grayguide}. It has been proved for attractive nearest-neighbor spin systems (a spin system is an interacting particle system with two states such that only one coordinate can change in each transition, see \cite{ligg:IPS}, Chapter II) in \cite{Grayrates}. But, even if the permutations have finite range, our model does 
not fit into this frame, since the cut-and-paste mechanism changes several sites at 
the same time, and moreover we work with an alphabet of more than two values.
 There is a general sufficient condition for ergodicity, the so-called Dobrushin's ``$M < \varepsilon$'' 
 condition (see \cite{dobrushin:M<eps}, or \cite{ligg:IPS}, Chapter I). 
However, this condition is far from being necessary: For instance, it gives 
ergodicity of the one-dimensional nearest-neighbor stochastic Ising model only for high 
values of the temperature. Note that if a process does fit the ``$M < \varepsilon$'' condition, 
adding a small amount of ``stirring''  will not alter the validity of the latter, hence the 
superimposition with a stirring mechanism  (which permutes the values on two neighboring sites) 
is ergodic for $\rho$ small enough.

To derive ergodicity conditions for our superimposition of two dynamics, 
we use two powerful techniques: the first one is based
on duality, and the second one on attractiveness and couplings. 
Our results involve only the parameters of the substitution process, 
thus they are valid for {\sl any} value of $\rho$.
First we use duality through a generalization of a coupling technique with branching processes 
due to \cite{ferrari:ErgoSpinStirring} to show that, in a certain range of parameters of the 
substitution process, the superimposition is ergodic.  We then apply this general result 
to our generic example, the RN+YpR substitution model with an additional cut-and-paste mechanism. 
Second, in the spirit of \cite{GS, borrello-1},
we derive sufficient conditions on the substitution rates for attractiveness of the superimposition; 
assuming them, we obtain by different approaches
sufficient conditions on theses rates for ergodicity.
 While our first results are general, the following ones concern
the RN+YpR  model with cut-and-paste
mechanism.  They imply that its particular cases,   
the  models JC+$\cc$p$\gg$, T92+$\cc$p$\gg$, 
as well as  the  RNc+YpR model in case it is attractive, are ergodic
when they are superimposed with a cut-and-paste
mechanism.  \\

The paper is organized as follows.  In Section \ref{sect:DefAndExa} 
 we define the model 
 and the examples,  for which we state ergodicity results 
 (Theorems \ref{th:erg-RN+YpR}  and \ref{th:ergo-first-exa}). 
 In Section \ref{sect:results-ergo_duality}, we give the set-up for generalized
 duality, then state general ergodicity results  
 (Theorems \ref{theo:Ergo} and \ref{theo:Ergo2}), proved in Section 
 \ref{sect:ProofsDuality}; we then prove Theorem \ref{th:erg-RN+YpR}. 
 In Section \ref{sect:results-AttrErg}, we give the set-up for
 attractiveness, then state various results (some general, some for our examples) 
 for ergodicity, proved in Section \ref{sect:ProofsAttrErg},
 and we prove Theorem \ref{th:ergo-first-exa}.   
\section{Definitions and examples} \label{sect:DefAndExa}
In  Section~\ref{sect:ourIPS}, we
introduce  our set-up: two  types of  interacting particle  systems on
$X=\AA^\Z$,  namely a   substitution  and  a cut-and-paste
processes,  that  we superimpose.
We then define in Section~\ref{sec:Examples} the examples that will illustrate our analysis along the paper.
In Section \ref{sec:statements-erg-examples}, we state for them the ergodicity results that follow from
our analysis of the superimposed process, done in the rest
of the paper. 
\subsection{The particle system} \label{sect:ourIPS}
For an analytic study of the construction and basic properties
of these systems, we refer to the seminal
book \cite{ligg:IPS},  on which we rely  for sufficient existence conditions
of our dynamics. We will give in Section \ref{sect:ProofsDuality} 
a graphical construction of the latter,
which will be the first step to prove ergodicity results through duality. 
\subsubsection{Substitution process} \label{sect:Subs}
In such a  process,  only   one  coordinate   changes   in  each
transition.  The transition
mechanism is specified  by a non-negative function $c(\cdot)$  defined on $\Z
\times \AA \times  X$. For  $\eta \in  X$,  $x  \in \Z$  and  $a  \in \AA$, 
$c(x,a,\eta)$ represents  the rate at
which  the coordinate $\eta(x)$  flips to  $a$ when  the system  is in
state $\eta$, that is the rate at which
$\eta$  changes   to  $\eta^x_a$ defined by 
\[
  \eta^x_a(z) = \left\{ \begin{array}{lll} \eta(z) & \mbox{if} & z \neq x,
       \\ a &  \mbox{if} & z = x. \end{array} \right. 
\]
We assume that the rates 
$c(x,a,\cdot)$ are translation invariant, i.e. $c(x,a,\eta) = c(0,a, \tau_x\eta)$ 
where $\tau_x$ denotes the shift by $x$ on $X$ 
(given by $\tau_x\eta(y) = \eta(x+y)$, for all $y \in \Z$).
{}From now on, we write $c(a,\eta)$ for $c(0,a, \eta)$.
We  further assume that, for any target $a$,
the function  $c(a,\cdot)$ depends on $\eta \in X$ only  through a
finite set  $S(a) \subset \Z$ depending on $a$.

 The  pregenerator  $\LL_1$   of  the
substitution process is defined on a cylinder 
function $f$ on $X$  (that is, a function depending on a finite number of coordinates) 
by 
\begin{equation} \label{equa:genesubs1}
  \LL_1  f(\eta)  = \sum_{x  \in  \Z}  \sum_{a  \in \AA}\  c(a,\tau_x \eta)
  \big[f(\eta^x_a) - f(\eta)\big].
\end{equation}
 Let
\begin{align} \label{equa:CondSpinM}
  m &= \inf\{ c(a,\eta) \, : \, a \in \AA, \eta \in X \},\\ 
  \label{equa:CondSpinK}
  K &=  \sup\{ c(a,\eta) \,  : \, a \in  \AA, \eta \in  X \},\\
  \label{equa:CondSpinS}
  s &= \sup\{|S(a)|\, : \, a \in \AA\},
\end{align}
where $|S|$ denotes the cardinality of  the set $S$. Because 
the alphabet $\AA$ is finite, we have that
\begin{equation}\label{equa:K-s-finite}
s<+\infty,\quad  K<+\infty.
\end{equation}
This is a sufficient condition  for the  existence of a Markov 
process with pregenerator $\LL_1$  (see Theorem I.3.9 in \cite{ligg:IPS}).
 For Theorem \ref{theo:Ergo} 
and for our examples, we will  assume
\begin{equation} \label{equa:Mpositif}
  m >0,
\end{equation}
which is a standard assumption for  ergodicity results  (see e.g. \cite{Grayguide}, \cite{Grayrates}). 
\subsubsection{Cut-and-paste process} \label{sect:Trans}
 In this process,  not  only one
coordinate  changes  in   each  transition
  (as  it  can   be  seen  on
Figure~\ref{figu:InsDel}   where  the  transfer   of  nucleotide~$\tt$
induces a change of four coordinates  because of the shift to the left
for the segment $ \cc \gg \cc$). 
The transition mechanisms are circular
permutations $\sigma_{x,y}$ (for $x\not=y \in \Z$) of  finitely many  sites of $\Z$  and are specified  by a
transition probability matrix  $p$ on $\Z \times \Z$. \\  
For $\eta \in  X$  and  a pair of sites  $x\neq y$,  let
$\sigma_{x,y} (\eta)$ be the configuration  defined by 
\[
  \sigma_{x,y} (\eta)(z) = \eta \big[\sigma_{x,y}^{-1}(z) \big], \quad
  \forall z \in \Z, 
\]
where $\sigma_{x,y}$ is defined for any $x<y$ as
\[
  \sigma_{x,y}(z) = \left\{ \begin{array}{lll} z & \mbox{if} & z
      \notin \{ x, x+1, \dots, y \}, \\ y &  \mbox{if} & z=x, \\
    z-1 &  \mbox{if} & x < z \leq y, \end{array} \right. 
\]
and for any $x > y$ as
\[
  \sigma_{x,y}(z) = \left\{ \begin{array}{lll} z & \mbox{if} & z
      \notin \{ y, y+1, \dots, x \}, \\ y &  \mbox{if} & z=x, \\
    z+1 &  \mbox{if} & y \leq  z < x. \end{array} \right. 
\]
The case $x<y$  corresponds to a shift to the  left of the coordinates
from  site  $x+1$ to  site  $y$  as  it can  be  seen  on the  top  of
Figure~\ref{figu:Shift}, whereas $x>y$ corresponds to a shift
to the right of the coordinates from   $y$ to  $x-1$ as it can
be seen on the bottom of Figure~\ref{figu:Shift}. \\
\begin{figure}[ht]
\begin{center}
\includegraphics{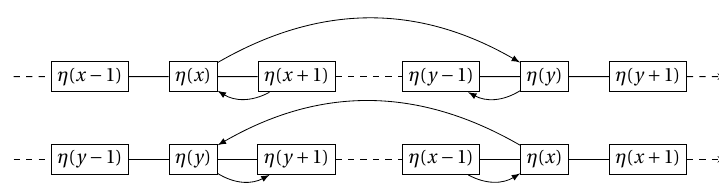}
\end{center}
\caption{Transfer of the coordinate $\eta(x)$  from site $x$ 
to site $y$ when $x<y$ (top panel) or when $x>y$ (bottom panel)} 
\label{figu:Shift}
\end{figure}
\noindent
The rate  at which $\eta$ changes  to $\sigma_{x,y}(\eta)$ is $p(x,y)$. 
We  assume   that  $p$  is translation invariant on $\Z$, that is, 
 $ p(x,y) = p(0,y-x)$  for all $x,y \in \Z$.
 We do not assume that
 $p$ is symmetric, hence the  ``rate of transfer'' $p(x,y)$ at  which  the  coordinate
$\eta(x)$  is   transferred  to  site~$y$  might depend
not only on the distance  $|y-x|$  but also on the direction of the transfer.\\ 
The pregenerator  $\LL_2$ of a
cut-and-paste process is defined on  a cylinder function  $f$ on  $X$ by 
\begin{equation} \label{equa:gene2}
  \LL_2 f (\eta) = \sum_{x,y \in \Z} p(x,y) \big[
  f(\sigma_{x,y}(\eta)) - f(\eta) \big].
\end{equation}
A sufficient condition for the existence of the cut-and-paste process, 
that we assume from now on,  is (see \cite{ligg:IPS}, \cite{andjelAl:LLNExclusion})
\begin{equation} \label{equa:TransCond2}
  \sum |x| p(0,x) < \infty.
\end{equation}
\begin{rema}\label{rk:kovchegov} 
(i) We
mentioned that $p$ is not necessarily symmetric
  because \cite{kovchegov} studies invariant measures of  particle
systems  that interact via  finite range  permutations, which  is a more
general system than a  cut-and-paste process, but with the restriction
that  the permutation  $\sigma$ has  the same  rate of  occurrence than
$\sigma^{-1}$,  which is  equivalent in  our case  to the  symmetry of
$p$. \\
Ergodicity of cellular automata that are superimpositions 
of Glauber dynamics and  permutations is
studied in \cite{ferrari:Automata}.  \\ 
(ii) For simplicity, we stick here to circular permutations. As it will become clear from the proofs, 
we could also consider more general permutations of 
$\eta(x), \eta(x+1), \ldots ,\eta(y)$ or $\eta(y), \eta(y+1), \ldots \eta(x)$, respectively, resulting in different biological interpretations than the one we gave in the introduction.\\
(iii) The cut-and-paste process generalizes a stirring process, which is such that
$p(.,.)$ is symmetric and nearest-neighbor, so that only two coordinates change in a transition.
\end{rema} 
\subsubsection{The superimposition}\label{sec:superimposed}
Fix a  constant $\rho > 0$ and define the pregenerator  $\LL$ on a
cylinder function $f$ on $X$ as 
\begin{equation} \label{equa:genesubs3}
\LL f = \LL_1 f+ \rho \LL_2 f .
\end{equation}
We  are  interested  in  the  ergodic properties  of  the  interacting
particle system with pregenerator  $\LL$, that is, the superimposition
of a substitution and a cut-and-paste process.  This superimposition is
well    defined    by    assumptions    \eqref{equa:K-s-finite}    and
\eqref{equa:TransCond2}.  We denote 
by  $\pp(X)$  the set  of probability measures on $X$,  
by $\SS$ the set of 
translation invariant probability
measures on $X$. For the superimposition we denote by $\II$ the set 
of invariant probability measures. 

Recall that (Definition I.1.9 in \cite{ligg:IPS}) ergodicity of a Markov process 
with values in $X=\AA^\Z$ means that there is exactly {\sl one} invariant 
probability  measure,  denoted e.g. by $\mu$,  and for each starting point, the law of the process 
converges to this invariant probability  measure.

 This ergodic  process is moreover \textit{exponentially ergodic}
(\cite{ferrari:ErgoSpinStirring}, page 1526) if for any bounded cylinder function $f$ on $X$,
there exist
positive constants $a_1 = a_1( f), a_2$ such that for any initial probability measure
$\nu$, for any $t\geq 0$, we have
\begin{equation}\label{def:exp-erg}
\left|\int fd(\nu S(t)) - \int fd\mu \right| < a_1 \ee^{-a_2 t}.
\end{equation} 

The simplest ergodic substitution processes are the
independent ones. As expected,  the superimposition
of a cut-and-paste  mechanism does not affect the  ergodic properties and
the invariant probability measure of such processes. 
In this case, we do not need assumption \eqref{equa:Mpositif}, and
moreover $s=1$ (see \eqref{equa:CondSpinS}). 
\begin{lemm} \label{theo:Inde}
Assume that the rate function $c(\cdot)$ can be written as
\begin{equation} \label{equa:IndeTaux}
  c(x,a,\eta) = q(\eta(x),a),
\end{equation}
where  $Q=(q(a,b))_{a,b \in  \AA}$ is the infinitesimal  generator of  an
irreducible  continuous-time   Markov  chain  on   $\AA$  with  unique
invariant probability measure $\pi$.
Then the process $(\eta_t)_{t\geq 0}$ with
generator  $\LL$ given  by~\eqref{equa:genesubs3} is  ergodic  and its
unique invariant probability measure is the product measure~$\pi^{\otimes \Z}$.
\end{lemm}
We omit the simple proof of this lemma.
The examples of substitution processes we will work with from now on are not 
independent ones, they will satisfy $s>1$. 
\subsection{Examples of substitution models}\label{sec:Examples}
In this section we first define  our generic example, the RN+YpR model, that was
introduced and studied in \cite{piau:solv}, to which we refer for biological motivation. 
According to the values of its parameters, this mathematical model contains 
many known biological situations: 
We define more and more particular cases of it, the models RNc+YpR, T92+$\cc$p$\gg$, and 
JC+$\cc$p$\gg$.
\subsubsection{The RN+YpR model} \label{exam:RN+YpR}
First, RN stands for Rzhetsky-Nei  \cite{RN}  and means that the $ 4\times4 $ matrix of
substitution rates which characterize the independent evolution of the
sites must satisfy $4$ equalities, summarized as follows: for every
pair of neighboring nucleotides $ a $ and $ b \neq a $, the substitution rate from
$a$ to  $b$ may depend on $a$ but only through the fact that $a$  is a
purine ($\aa$ or $\gg$, symbol $R$) or a pyrimidine ($\cc$ or $\tt$, symbol
$Y$). For instance, the substitution rates from $\cc$ to $\aa$ and from
$\tt$ to $\aa$ must coincide,  as well as from $\aa$
and  from $\gg$ to $\cc$, from $\cc$  and from $\tt$
to $\gg$, and 
finally from $\aa$  and from $\gg$ to $\tt$. The $4$ remaining rates,
corresponding to purine-purine  
 and to pyrimidine-pyrimidine substitutions, are free. The 
 matrix  of  substitution rates is given by 
\[
  \bordermatrix{
      & \aa     &\tt      & \cc    &\gg      \cr
    \aa & \cdot & v_\tt   & v_\cc   & w_\gg   \cr
    \tt & v_\aa   & \cdot & w_\cc   & v_\gg   \cr
    \cc & v_\aa   & w_\tt   & \cdot & v_\gg   \cr
    \gg & w_\aa  & v_\tt & v_\cc & \cdot  \cr},  \quad \mbox{with}
  \quad 0 \leq v_a \leq w_a 
\quad \mbox{for all} \quad  a\in\AA. 
\]
Second, the influence mechanism is called YpR, which stands for the
fact that one allows any specific substitution rate between any two
YpR  dinucleotides ($\cc  \gg$, $\cc  \aa$, $\tt  \gg$ and  $\tt \aa$)
for a total of $8$ independent parameters. Hence, every
dinucleotide $\cc \gg$   moves to $\cc \aa$ at  rate $r_\aa^\cc$  and to
$\tt \gg$ at  rate $r_\tt^\gg$; every dinucleotide $\tt  \aa$ moves to
$\cc \aa$ at rate $r_\cc^\aa$ and to 
  $\tt \gg$ at rate $r_\gg^\tt$; every dinucleotide $\cc \aa$ moves to
  $\cc \gg$ at rate $r_\gg^\cc$  and to $\tt \aa$ at rate $r_\tt^\aa$;
  every dinucleotide $\tt \gg$  moves to $\cc \gg$ at rate~$r_\cc^\gg$
  and to $\tt \aa$ at rate $r_\aa^\tt$.  
We have 
\[
  S(\aa)=\{-1,0\} \quad \mbox{and} \quad
  c(\aa,\eta)= \left\{ 
  \begin{array}{ll}
  v_\aa & \mbox{if }  \eta(0) \in \{\cc,\tt\}=Y, \\
  w_\aa & \mbox{if } \eta(0) =  \gg \mbox{ and }
     \eta(-1) \notin \{\cc,\tt\}, \\ 
  w_\aa  + r_\aa^a  & \mbox{if } \eta(0)  = \gg
     \mbox{ and } \eta(-1)=a\in \{\cc,\tt\}, 
  \end{array}
  \right.
\]
\[
  S(\tt)=\{0,1\} \quad \mbox{and} \quad
  c(\tt,\eta)= \left\{ 
  \begin{array}{ll}
  v_\tt & \mbox{if }  \eta(0) \in \{\aa,\gg\}=R, \\
  w_\tt & \mbox{if } \eta(0) =  \cc \mbox{ and }
     \eta(1) \notin \{\aa,\gg\}, \\
      w_\tt  + r_\tt^a & \mbox{if  } \eta(0)  = \cc
     \mbox{ and } \eta(1)=a\in \{\aa,\gg\}, 
      \end{array}
  \right.
\]
\[
  S(\cc)=\{0,1\} \quad \mbox{and} \quad
  c(\cc,\eta)= \left\{
  \begin{array}{ll}
  v_\cc & \mbox{if }  \eta(0) \in \{\aa,\gg\}=R, \\
  w_\cc & \mbox{if } \eta(0) =  \tt \mbox{ and }
     \eta(1) \notin \{\aa,\gg\}, \\ 
  w_\cc  + r_\cc^a & \mbox{if } \eta(0)  = \tt
     \mbox{ and } \eta(1)=a\in \{\aa,\gg\},
   \end{array}
   \right. 
\]
\[
  S(\gg)=\{-1,0\} \quad \mbox{and} \quad
  c(\gg,\eta)= \left\{
  \begin{array}{ll}
  v_\gg & \mbox{if }  \eta(0) \in \{\cc,\tt\}=Y, \\
  w_\gg & \mbox{if } \eta(0) =  \aa \mbox{ and }
     \eta(-1) \notin \{\cc,\tt\}, \\ 
  w_\gg  + r_\gg^a & \mbox{if } \eta(0)  = \aa
     \mbox{ and } \eta(-1)=a\in \{\cc,\tt\}.
\end{array}
\right.
\]
\begin{equation}\label{equa:kms}
K=\max\big\{w_a+r_a^b,w_b+r_b^a \,: \, a \in  Y, b \in R \big\}, \quad
m=\min_{a\in \AA} v_a, \quad \mbox{and} \quad s=2. 
\end{equation}
 We assume from now on that \eqref{equa:Mpositif} is satisfied, 
that is $\min_{a\in \AA} v_a>0$. 
 \subsubsection{The RNc+YpR model} \label{exam:RNc+YpR}
In this model, see \cite{piau:solv} page 79, the substitution rates respect the ``strand
  complementarity of nucleotides'', so that the rates of YpR substitutions from $\cc \gg$ to $\cc \aa$
  and to $\tt \gg$ coincide,  from $\tt \aa$ to $\cc \aa$
  and to $\tt \gg$ coincide, from $\cc \aa$ and from $\tt \gg$ to $\cc \gg$ coincide, from $\cc \aa$ and from $\tt \gg$ to $\tt \aa$ coincide. Therefore 
\begin{eqnarray}\label{rates:RNc+YpR-1}
w_\aa=w_{\textsc w};\,v_\aa=v_{\textsc w} &;& r_\aa^\tt=r_{\textsc u};\,
r_\aa^\cc=r_{\textsc w},\\\label{rates:RNc+YpR-2}
w_\tt=w_{\textsc w};\,v_\tt=v_{\textsc w} &;& r_\tt^\aa=r_{\textsc u};\,r_\tt^\gg=r_{\textsc w},\\\label{rates:RNc+YpR-3}
w_\cc=w_{\textsc s};\,v_\cc=v_{\textsc s} &;& r_\cc^\aa=r_{\textsc s};\,r_\cc^\gg=r_{\textsc v},\\\label{rates:RNc+YpR-4}
w_\gg= w_{\textsc s};\,v_\gg=v_{\textsc s} &;& r_\gg^\tt=r_{\textsc s};\, r_\gg^\cc=r_{\textsc v}.
\end{eqnarray}
 \subsubsection{The T92+$\cc$p$\gg$ model}
 \label{exam:T92+CpG}
This model adds neighboring effects to the classical T92 model developed by Tamura in \cite{tamu:92},
 which consisted in an independent evolution of the sites. 
  The values of the above rates \eqref{rates:RNc+YpR-1}--\eqref{rates:RNc+YpR-4}  
  for the T92+$\cc$p$\gg$ model are,  for some $\theta\in[0,1]$, 
\begin{eqnarray}\label{rates:T92+CpG-1}
w_\aa=(1-\theta)w;\,v_\aa=(1-\theta)v &;& r_\aa^\tt=0;\,r_\aa^\cc=r,\\\label{rates:T92+CpG-2}
w_\tt=(1-\theta)w;\,v_\tt=(1-\theta)v &;& r_\tt^\aa=0;\,r_\tt^\gg=r,\\\label{rates:T92+CpG-3}
w_\cc=\theta w;\,v_\cc=\theta v &;& r_\cc^\aa=r_\cc^\gg=0,\\\label{rates:T92+CpG-4}
w_\gg=\theta w;\,v_\gg=\theta v &;& r_\gg^\tt=r_\gg^\cc=0. 
\end{eqnarray}
\subsubsection{The JC+$\cc$p$\gg$ model}
 \label{exam:JC+CpG}
Again this model adds neighboring effects to the JC model (see Section \ref{sec:intro}).
 The values of the above rates for the JC+$\cc$p$\gg$ model correspond 
 to the doubled values of the ones of the T92+$\cc$p$\gg$ model 
 for $\theta=1/2$: 
\begin{eqnarray}\label{rates:JC+CpG-1}
w_\aa=v_\aa=v &;& r_\aa^\tt=0;\,r_\aa^\cc=r,\\\label{rates:JC+CpG-2}
w_\tt=v_\tt=v &;& r_\tt^\aa=0;\,r_\tt^\gg=r,\\\label{rates:JC+CpG-3}
w_\cc=v_\cc=v &;& r_\cc^\aa=r_\cc^\gg=0,\\\label{rates:JC+CpG-4}
w_\gg=v_\gg=v &;& r_\gg^\tt=r_\gg^\cc=0.
\end{eqnarray}
\subsection{Ergodicity results for these substitution models 
with additional cut-and-paste mechanism}\label{sec:statements-erg-examples}
Assuming that $\min_{a     \in     \AA}     v_a>0$ 
 (that is, \eqref{equa:Mpositif}),  
\cite{piau:solv}   proved   that   the   RN+YpR   model, 
and as a consequence  the RNc+YpR, T92+$\cc$p$\gg$ and 
JC+$\cc$p$\gg$  models are ergodic for all     
substitution rates. 
As pointed out in  \cite[Theorem 6]{piau:solv}, considering only the evolution
of $Y$ and $R$ 
instead of the one of the four elements
of $\AA$ gives that the (only) invariant probability measure is a product measure.\\
Our main results  for these examples are the following.
\begin{theo}\label{th:erg-RN+YpR} 
For any $\rho > 0$,  the RN+YpR model with cut-and-paste mechanism  
is exponentially ergodic (recall \eqref{def:exp-erg}) if  
\begin{equation} \label{equa:RNergo}
  \min_{a     \in     \AA}     v_a>0    \quad     \mbox{and}     \quad
  \max\big(\YY\cup\RR   \big)  
  <  \sum_{a\in \AA} v_a, 
 \end{equation} 
 where
 \begin{eqnarray}\label{def:calY}
 \YY&=& \{ r_a^\aa\vee r_a^\gg-r_a^\aa\wedge r_a^\gg,\,r_a^\aa\wedge r_a^\gg\,:\, a \in  Y\},\\ \label{def:calR}
 \RR&=& \{  r_b^\cc\vee r_b^\tt-r_b^\cc\wedge r_b^\tt,r_b^\cc\wedge r_b^\tt\,:\, b \in  R \}.
 \end{eqnarray}  
 If the second inequality in \eqref{equa:RNergo} is an equality, then
the RN+YpR model with cut-and-paste mechanism  is ergodic. 
\end{theo}
We will prove Theorem \ref{th:erg-RN+YpR}  in Section \ref{sect:appl-duality_RN+YpR}, 
using the duality technique 
introduced and developed in Subsections \ref{sect:first-ergo_duality} and 
\ref{sect:second-ergo_duality}  for general substitution processes
with cut-and-paste mechanism. 
 The next corollary is a direct application
of Theorem \ref{th:erg-RN+YpR}.
\begin{coro}\label{cor:dual-result-RNc+YpR}
Assume that the rates of the RNc+YpR model 
(defined in \eqref{rates:RNc+YpR-1}--\eqref{rates:RNc+YpR-4}) satisfy 
\begin{eqnarray}
 \max\big(r_{\textsc u}\vee r_{\textsc w}-r_{\textsc u}\wedge r_{\textsc w},\, 
 r_{\textsc u}\wedge r_{\textsc w},\,
 r_{\textsc s}\vee r_{\textsc v}-r_{\textsc s}\wedge r_{\textsc v},\, 
 r_{\textsc s}\wedge r_{\textsc v} \big)
 \leq 2 (v_{\textsc s}+v_{\textsc w}). \label{eq:dual-rates-RNc+YpR}
\end{eqnarray}
Then, for any $\rho > 0$,
the RNc+YpR model  with cut-and-paste mechanism is ergodic.
\end{coro}
The attractiveness and coupling techniques will be
introduced and developed in Subsection \ref{sect:first-AttrErg}  
for general substitution processes with cut-and-paste mechanism, 
and specialized in Subsection \ref{sec:ErgoAttrRN+YpR} to the RN+YpR model.
We will prove Theorem \ref{th:ergo-first-exa}    
in Section 
\ref{sec:ErgoAttrRN+YpR}. 
\begin{theo}\label{th:ergo-first-exa} 
For any $\rho > 0$, we have the following.
\begin{itemize}
\item[(i)] The T92+$\cc$p$\gg$ model, and as a consequence the JC+$\cc$p$\gg$ model,
both  with cut-and-paste mechanism, 
are ergodic for all substitution rates.
\item[(ii)]   Assume that the rates of the RNc+YpR model 
(defined in \eqref{rates:RNc+YpR-1}--\eqref{rates:RNc+YpR-4}) satisfy 
\mbox{attractiveness conditions}, namely
\begin{eqnarray}\label{eq:attra-rates-RNc+YpR-1}
\mbox{either}&&
r_{\textsc u}\leq r_{\textsc w};\,r_{\textsc s}=r_{\textsc v}=0, \\ \label{eq:attra-rates-RNc+YpR-2}
 \mbox{or}&& r_{\textsc s}\leq r_{\textsc v};\,r_{\textsc u}=r_{\textsc w}=0. \label{eq:attra2-rates-RNc+YpR}
\end{eqnarray}
Then, the RNc+YpR model  with cut-and-paste mechanism is ergodic.
\end{itemize}
\end{theo}
 More precisely, we will derive various sets of 
conditions for the substitution rates of the processes which imply ergodicity 
(Propositions \ref{prop:mu-a-t-c-g}, \ref{prop:mu-a-t_or_c-g}, \ref{prop:nu-diag}), 
and we will apply them to derive Theorem \ref{th:ergo-first-exa}.  
%
%
\section{Ergodicity through generalized duality} \label{sect:results-ergo_duality}
The starting point of this approach is a graphical construction of the dynamics,
using a Harris representation \cite{harris},  done in Section \ref{sect:Graph}.
To state our results in Section \ref{sect:second-ergo_duality}, we have to introduce 
the necessary notation in Section \ref{sect:first-ergo_duality}. 
\subsection{Set-up}\label{sect:first-ergo_duality} 
In the pregenerator $\LL_1$ defined by~\eqref{equa:genesubs1}, 
for each $a \in \AA$, write the substitution
rate $c(a,\eta)$ depending on the finite set $S(a) \subset \Z$ as
\begin{equation} \label{equa:SpinFlipRate}
  c(a,\eta)  =  \sum_{j \in  J(a)}  \lambda_j(a)  \1_{\{ \eta  \in
  A_j(a) \}}, 
\end{equation}  
where  $\lambda_{j+1}(a) \geq  \lambda_j(a)$ and  $A_j(a)$  are
cylinder sets of $X$ depending on $S_j(a)\subseteq S(a)$ such that the family 
$\{A_j(a)\}_{j \in J(a)}$ is a partition of $X$.  Thus,
\begin{equation} \label{equa:partition}
  \sum_{j \in J(a)} \1_{\{ \eta \in A_j(a)\}} = 1.
\end{equation}
By convention, the first label in the set $J(a) \subset
\N$ is $0$. The number of elements of 
$J(a)$  is uniformly bounded by  $s^{|\AA|}$, where $s$ was defined
by~\eqref{equa:CondSpinS}.   Indeed, since the  rate $c(a,\eta)$
depends at most on $s$ sites, there are at most 
$s^{|\AA|}$  different cylinder  sets $A_j(a)$.  
We set   
\begin{eqnarray} \label{equa:blambda}
    \blambda_j(a)&=&  \lambda_j(a)  -   \lambda_{j-1}(a)  \quad
  \mbox{for}\,\,  j  \in J(a)\setminus\{0\}, \,\,
  \mbox{and} \quad \blambda_{0}(a)=\lambda_{0}(a),  
\\ \label{equa:Lambda}
  \lambda (a)&=&\max\big\{\lambda_j(a) \, : \, j \in J(a) \big\} -
  \lambda_0(a) = \lambda _{|J(a)| -1 }- \lambda_0(a), 
\end{eqnarray}
and
\begin{equation} 
\label{equa:bLambda0}
   \blambda_0 = \sum_{a  \in  \AA} \blambda_0(a).
\end{equation}
We  first assume~\eqref{equa:Mpositif}; then  the  latter
quantity is positive since $ \blambda_0(a) \geq m$. Finally, set
 \begin{equation} \label{equa:branch_sup-sup}
  \blambda =  \max_{a\in \AA}  \max_{j \in  J(a)
    \setminus \{0\}}  \blambda_j (a).
\end{equation} 
\subsection{Results}\label{sect:second-ergo_duality}
\begin{theo}\label{theo:Ergo}
Assume \eqref{equa:Mpositif} and that 
\begin{equation}\label{equa:ErgoFin}
  (s-1)  \blambda <  \blambda_0.
\end{equation}
Then, 
for any $\rho > 0$, the process $(\eta_t)_{t\geq 0}$ with
generator  $\LL$   given  by~\eqref{equa:genesubs3}  is  exponentially
ergodic.  If $(s-1) \overline \lambda = \overline \lambda_0$  then  the
process is ergodic.  
\end{theo}
 \begin{rema} \label{rema:assu-1}
(i) 
  A stronger condition for exponential ergodicity than 
\eqref{equa:ErgoFin} is that $m,K,s$          
(given by~\eqref{equa:CondSpinM}--\eqref{equa:CondSpinS})  satisfy
\begin{equation} \label{equa:Ergo}
  m >0 \quad \mbox{and} \quad (s-1)(K-m) < |\AA|m.
\end{equation}
This last inequality is the natural extension of the one in 
 \cite[Theorem  2.1]{ferrari:ErgoSpinStirring},
done for a two-letter alphabet. \\ 
 (ii) Condition \eqref{equa:ErgoFin} is a priori non trivial 
if $s>1$ and $|J(a)|>1$ for some $a\in\AA$.\\   
\end{rema} 
\noindent\textbf{{Refinement of Theorem~\ref{theo:Ergo}}.} 
Assume  that $\LL_1$ 
can be decomposed  as a sum of  several generators
\begin{equation} \label{equa:genesubs4}
  \LL_1 = \sum_{i=1}^d\LL_1^{(i)}, \quad \mbox{with} \quad
  \LL_1^{(i)}  f(\eta)  = \sum_{x  \in  \Z}  \sum_{a  \in \AA}\  c^{(i)}(a,\tau_x \eta)
  \big[f(\eta^x_a) - f(\eta)\big].
\end{equation}
For any  $i$ in  $\{1, \dots, d  \}$, define $m^{(i)},\,K^{(i)},\,s^{(i)}$     
as     in~\eqref{equa:CondSpinM}--\eqref{equa:CondSpinS},  
replacing $c(\cdot, \cdot)$  by $c^{(i)}( \cdot, \cdot)$,
defined through  $J^{(i)}(a)$, $\lambda_j^{(i)}(a)$,
 $A_j^{(i)}(a)$  as in \eqref{equa:SpinFlipRate}.
 In particular we have
\begin{equation} \label{equa:partition-i}
  \sum_{j \in J^{(i)}(a)} \1_{\{ \eta \in A_j^{(i)}(a)\}} = 1.
\end{equation} 
 \begin{rema} \label{rema:decomp_i}
Such a decomposition of the generator might
be useful if, for some $i$ in $\{1,\dots,d\}$, 
the rate  $c^{(i)}(\cdot, \cdot)$ depends on less 
sites than the  original rate $c(\cdot, \cdot)$, 
that is, $s^{(i)} < s$, or if $s^{(i)}=1$. See Theorem \ref{theo:Ergo2}
below.
\end{rema}
 As in \eqref{equa:blambda}, 
we set, for $i$ in $\{1,\dots,d\}$, 
\begin{eqnarray} \label{equa:blambda_i}
\blambda_j^{(i)}(a)&=& \lambda_j^{(i)}(a)-\lambda_{j-1}^{(i)}(a)
\quad \mbox{for} \,\, j \in J^{(i)}(a)\setminus\{0\}, \,\,      
\mbox{and} \quad  \blambda_{0}^{(i)}(a)=\lambda_{0}^{(i)}(a),  
\\ \label{equa:Lambda-i-total}
  \lambda^{(i)} (a)&=&\max\big\{\lambda_j^{(i)}(a) \, : \, j \in J^{(i)}(a) \big\} -
  \lambda_0^{(i)}(a) = \lambda^{(i)} _{|J^{(i)}(a)|-1}- \lambda^{(i)}_0(a).
\end{eqnarray} 
We replace assumption~\eqref{equa:Mpositif} by
\begin{equation} \label{equa:Mpositif2}
  \sum_{i=1}^d m^{(i)} >0,
\end{equation}
which can  be true  whereas for  some $i$, $m^{(i)}=0$. 
According to  \eqref{equa:bLambda0},  we would now define
$\blambda_0^{(i)} = \sum_{a \in \AA}  \blambda_0^{(i)}(a),$
but    this   quantity    could   be  zero   for    some    $i$   in
$\{1,\dots,d\}$. Therefore, the right quantities to consider are 
\begin{equation} \label{equa:branch_inf_i}
   \blambda_{0,d}(a) = \sum_{i=1}^d  \blambda_0^{(i)}(a)  
  \quad  \mbox{and} \quad  \blambda_{0,d}  = \sum_{a
    \in \AA}  \blambda_{0,d}(a), 
\end{equation}
which    are    positive    by~\eqref{equa:Mpositif2}   since    
$\blambda_{0,d}(a) \geq \sum_{i=1}^d m^{(i)}$. 

Finally, we set
 \begin{equation} \label{equa:branch_sup-sup_i}
    \blambda^{(i)} =  \max_{a\in \AA}  \max_{j \in  J^{(i)}(a)
    \setminus \{0\}}  \blambda_j^{(i)} (a).
\end{equation} 
Then we have 
\begin{theo} \label{theo:Ergo2}
Assume  \eqref{equa:Mpositif2} and  that
\begin{equation}
	\label{equa:ErgoFin2}
  \sum_{i=1}^d (s^{(i)}-1)  \blambda^{(i)} <  \blambda_{0,d}.
\end{equation}
Then, for any $\rho > 0$, the process $(\eta_t)_{t\geq 0}$ with
generator  $\LL$   given  by~\eqref{equa:genesubs3}  is  exponentially
ergodic.  If $\sum_{i=1}^d (s^{(i)}-1)  \blambda^{(i)} =  \blambda_{0,d}$  then  the
process is ergodic. 
\end{theo}
\begin{rema} \label{rema:assu-bis}
(i)   A stronger condition for exponential ergodicity than \eqref{equa:ErgoFin2} is 
that   $(K^{(i)},s^{(i)},m^{(i)})_{i=1}^{d}$  satisfy
\begin{equation} \label{equa:Ergo2}
  \sum_{i=1}^{d}     m^{(i)}    >     0     \quad    \mbox{and}     \quad
  \sum_{i=1}^{d}(s^{(i)}-1)(K^{(i)}-m^{(i)}) < |\AA|\sum_{i=1}^{d} m^{(i)}. 
\end{equation}
 As in Remark \ref{rema:assu-1}, this last inequality is a natural extension 
of the one in  \cite[Theorem  2.2]{ferrari:ErgoSpinStirring}. \\ 
(ii) It is shown in \cite{ferrari:ErgoSpinStirring} in the case of 
a two-letter alphabet that \eqref{equa:ErgoFin2} improves, 
even without stirring, the usual
``$M<\varepsilon$'' condition for ergodicity. 
\end{rema} 
\subsection{Application to the RN+YpR model with 
cut-and-paste mechanism} \label{sect:appl-duality_RN+YpR}
\begin{proof} \textit{(of Theorem \ref{th:erg-RN+YpR}).} 
This is an application of Theorem \ref{theo:Ergo2}. 
The required notation for the RN+YpR model are contained in Table~\ref{tabl:RN+YpR}
(which refers to Section \ref{exam:RN+YpR}). 
\begin{table}[ht]
\caption{Values of the quantities introduced in
  Section~\ref{sect:DefAndExa} for the RN+YpR model}
  \[
  \renewcommand{\arraystretch}{1.2}
  \begin{array}{|c|c|c|c|c|c|c|}
    \hline
    a  &   S(a)  &   j  \in  J(a) & S_j(a) &  A_j(a)  &   \lambda_j(a)  &
    \blambda_j(a)\\ 
    \hline
    \aa & \{-1,0\} &  0 & \{0\} &\{ \eta \, : \, \eta(0) \in  Y \} & v_\aa &
    v_\aa\\ 
    &  &  1 & \{-1,0\} & \{  \eta \, :  \, \eta(-1)\eta(0)=b\gg, b\in  R \} &  w_\aa &
    w_\aa - v_\aa\\ 
    &  &  2 & \{-1,0\} & \{ \eta  \, : \, \eta(-1)\eta(0)=d\gg,  d \in Y,
    r^d_\aa = r^\cc_\aa \wedge r^\tt_\aa    \}  & w_\aa + r^d_\aa & r^d_\aa\\
    &  &  3 & \{-1,0\} & \{ \eta  \, : \, \eta(-1)\eta(0)=e\gg,  e \in Y,
    r^e_\aa = r^\cc_\aa \vee r^\tt_\aa  \}  & w_\aa +
    r^e_\aa & r^e_\aa - r^d_\aa\\ 
    \hline
    \tt & \{0,1\} &  0 & \{0\} & \{ \eta \, : \, \eta(0) \in  R \} & v_\tt &
    v_\tt \\ 
    &  &  1 & \{0,1\} & \{ \eta  \, : \, \eta(0)\eta(1)=\cc b, b\in Y \}  & w_\tt &
    w_\tt - v_\tt\\ 
    &  & 2 & \{0,1\}  & \{ \eta \, : \,  \eta(0)\eta(1)=\cc d, d \in R,
    r^d_\tt =  r^\aa_\tt\wedge r^\gg_\tt \} & w_\tt + r^d_\tt & r^d_\tt \\ 
    &  & 3 & \{0,1\}  & \{ \eta \, : \,  \eta(0)\eta(1)=\cc e, e \in R,
    r^e_\tt =  r^\aa_\tt\vee r^\gg_\tt \} & w_\tt + r^e_\tt & r^e_\tt - r^d_\tt \\ 
    \hline
    \cc & \{0,1\} &  0 & \{0\} & \{ \eta \, : \, \eta(0) \in  R \} & v_\cc &
    v_\cc\\ 
    &  &  1 & \{0,1\} & \{ \eta  \, : \, \eta(0)\eta(1)=\tt b, b\in Y \}  & w_\cc &
    w_\cc - v_\cc \\ 
    &  & 2 & \{0,1\} & \{ \eta \, : \,  \eta(0)\eta(1)=\tt d, d \in R,
    r^d_\cc = r^\aa_\cc\wedge r^\gg_\cc  \} & w_\cc + r^d_\cc & r^d_\cc \\ 
    &  & 3 & \{0,1\} & \{ \eta \, : \,  \eta(0)\eta(1)=\tt e, e \in R,
    r^e_\cc = r^\aa_\cc\vee r^\gg_\cc  \} & w_\cc + r^e_\cc & r^e_\cc - r^d_\cc\\ 
    \hline
    \gg & \{-1,0\} &  0 & \{0\} & \{ \eta \, : \, \eta(0) \in  Y \} & v_\gg &
    v_\gg \\ 
    &  &  1 & \{-1,0\} & \{  \eta \, :  \, \eta(-1)\eta(0)=b\aa, b\in  R \} &  w_\gg &
    w_\gg - v_\gg\\ 
    &  &  2 & \{-1,0\} & \{ \eta  \, : \, \eta(-1)\eta(0)=d\aa,  d \in Y, 
    r^d_\gg = r^\cc_\gg \wedge r^\tt_\gg\}  & w_\gg + r^d_\gg & r^d_\gg\\
    &  &  3 & \{-1,0\} & \{ \eta  \, : \, \eta(-1)\eta(0)=e\aa,  e \in Y, 
    r^e_\gg = r^\cc_\gg \vee r^\tt_\gg\}  & w_\gg + r^e_\gg & r^e_\gg - r^d_\gg\\ 
    \hline
    \hline
    \multicolumn{7}{|c|}{s=2,\quad 
     \blambda = \max\big(\{w_a-v_a\,: \, a \in \AA\} \cup 
    \YY\cup\RR \big), \quad \mbox{and}\quad 
      \blambda_0 = \sum_{a \in \AA} v_a}\\
    \hline
  \end{array}
  \]
\label{tabl:RN+YpR}
\end{table}
 We thus have that
  condition~\eqref{equa:ErgoFin} is equivalent to  (recall \eqref{def:calY}, \eqref{def:calR})
  \begin{equation}\label{equa:ErgoFin-IV}
  \min_{a \in \AA} v_a>0 \quad\mbox{and}\quad
    \max\big(\{w_a-v_a\,:  a \in \AA\}\cup\YY\cup\RR 
    \big)< \sum_{a \in \AA} v_a, 
  \end{equation} 
   while condition~\eqref{equa:Ergo} is equivalent to
  \[
  \min_{a \in \AA} v_a>0 \quad\mbox{and}\quad
  \max\big\{w_a+r_a^b,w_b+r_b^a  \,: \,  a \in  Y,  b \in  R \big\}  -
   \min_{a\in \AA} v_a <    4 \min_{a\in \AA} v_a. 
  \]
 One  can write
   $\LL_1=\LL_1^{(1)} + \LL_1^{(2)}$  with the rates  
 \[
  c^{(1)}(\aa,\eta)= \left\{ 
  \begin{array}{ll}
   r_\aa^a & \mbox{if }  \eta(0)  = \gg
     \mbox{ and } \eta(-1)=a\in \{\cc,\tt\}, \\
  0 & \mbox{else,}
  \end{array}
  \right.
\]
\[
  c^{(1)}(\tt,\eta)= \left\{ 
  \begin{array}{ll}
   r_\tt^a & \mbox{if  } \eta(0)  = \cc
     \mbox{ and } \eta(1)=a\in \{\aa,\gg\}, \\
  0 & \mbox{else},
      \end{array}
  \right.
\]
\[
  c^{(1)}(\cc,\eta)= \left\{
  \begin{array}{ll}
   r_\cc^a & \mbox{if } \eta(0)  = \tt
     \mbox{ and } \eta(1)=a\in \{\aa,\gg\}, \\
  0 & \mbox{else},
   \end{array}
   \right. 
\]
\[
  c^{(1)}(\gg,\eta)= \left\{
  \begin{array}{ll}
   r_\gg^a & \mbox{if } \eta(0)  = \aa
     \mbox{ and } \eta(-1)=a\in \{\cc,\tt\}. \\
  0 & \mbox{else},
\end{array}
\right.
\]
and
\[
  c^{(2)}(a,\eta)= \left\{
  \begin{array}{ll}
  w_a & \mbox{if } \{a,\eta(0)\}=\{\aa,\gg\}\mbox{ or }\{a,\eta(0)\}=\{\tt,\cc\}, \\
  v_a & \mbox{else.}
\end{array} \right.
\]
 (see the next two tables).
As a consequence, we have
 \begin{align}\label{eq:K-m-s_1}
  K^{(1)} &= \max\{r_a^b, r_b^a \, : \, a  \in R, b \in Y
  \}, \quad m^{(1)}= 0, \quad s^{(1)}= 2,\\ \label{eq:K-m-s_2}
  K^{(2)} &= \max_{a \in \AA} w_a, \quad m^{(2)}= \min_{a \in \AA} v_a, \quad s^{(2)}= 1.
\end{align} 
As claimed  in Remark~\ref{rema:decomp_i}, we  are in an interesting case because   $s^{(2)}=1 < s=2$. 
    \[
  \begin{array}{|c|c|c|c|c|c|}
    \hline
    a  &   S^{(1)}(a)  &  j   \in  J^{(1)}(a)  &   A_j^{(1)}(a)  &
    \lambda_j^{(1)}(a) & \blambda_j^{(1)}(a) \\ 
    \hline
    \aa &  \{-1,0\} & 0 &  \{ \eta \, :  \, (\eta(-1),\eta(0))\notin Y\times\{\gg\}
     \} & 0 & 0 \\ 
    & & 1 & \{ \eta \, : \,  \eta(-1)\eta(0)= d \gg, \, d \in Y, 
    r_\aa^d=r_\aa^\cc\wedge r_\aa^\tt \} &  r_\aa^d & r_\aa^d \\ 
    & & 2 & \{ \eta \, : \,  \eta(-1)\eta(0)= e \gg, \, e \in Y,
    r_\aa^e=r_\aa^\cc\vee r_\aa^\tt \} & r_\aa^e & r_\aa^e - r_\aa^d \\ 
    \hline
    \tt &  \{0,1\} & 0 &  \{ \eta \, :  \, (\eta(0),\eta(1))\notin \{\cc\}
    \times R \} & 0 & 0 \\ 
    & & 1 &  \{ \eta \, : \, \eta(0)\eta(1)=\cc d, \, d  \in R, 
    r_\tt^d=r_\tt^\aa\wedge r_\tt^\gg \} &  r_\tt^d & r_\tt^d \\ 
    & & 2 &  \{ \eta \, : \, \eta(0)\eta(1)=\cc e, \, e  \in R, 
    r_\tt^e=r_\tt^\aa\vee r_\tt^\gg \} &  r_\tt^e & r_\tt^e - r_\tt^d \\ 
    \hline
    \cc & \{0,1\} & 0 & \{ \eta \, : \, (\eta(0),\eta(1))\notin \{\tt \} 
    \times R \} & 0 & 0 \\ 
    & & 1 & \{ \eta \, : \, \eta(0)\eta(1)=\tt d, \, d \in R, 
    r_\cc^d = r_\cc^\aa\wedge r_\cc^\gg  \} & r_\cc^d & r_\cc^d \\
    & & 2 & \{ \eta \, : \, \eta(0)\eta(1)=\tt e, \, e \in R, 
    r_\cc^e = r_\cc^\aa\vee r_\cc^\gg  \} & r_\cc^e & r_\cc^e - r_\cc^d \\
    \hline
    \gg &  \{-1,0\} & 0 &  \{ \eta \, :  \, (\eta(-1),\eta(0))\notin Y\times\{\aa\}
     \} & 0 & 0 \\ 
    & & 1 & \{ \eta \, : \, \eta(-1)\eta(0)= d \aa, \, d \in Y, 
    r_\gg^d=r_\gg^\cc\wedge r_\gg^\tt \} &  r_\gg^d & r_\gg^d \\
    & & 2 & \{ \eta \, : \, \eta(-1)\eta(0)= e \aa, \, e \in Y, 
    r_\gg^e=r_\gg^\cc\vee r_\gg^\tt \} &  r_\gg^e & r_\gg^e - r_\gg^d \\
    \hline
        \hline
    \multicolumn{6}{|c|}{s^{(1)}= 2,\quad 
    \blambda^{(1)} = \max(\YY\cup\RR), 
    \quad \mbox{and}
      \quad 
      \blambda_0^{(1)} = 0}\\
    \hline
  \end{array}
  \]  
   \[
  \begin{array}{|c|c|c|c|c|c|}
    \hline
    a  & S^{(2)}(a)  & J^{(2)}(a) & A_j^{(2)}(a) & \lambda_j^{(2)}(a) & \blambda_j^{(2)} \\ 
    \hline
    \aa &  \{0\} & 0 &  \{ \eta \, :  \, \eta(0)\neq \gg \} & v_\aa & v_\aa \\ 
    & & 1 & \{ \eta \, : \, \eta(0)= \gg \} & w_\aa  & w_\aa - v_\aa \\
    \hline
    \tt & \{0\} & 0 & \{ \eta \, : \, \eta(0)\neq \cc \} & v_\tt & v_\tt \\ 
    & & 1 & \{ \eta \, : \, \eta(0)=\cc \} & w_\tt  & w_\tt - v_\tt \\
    \hline
    \cc & \{0\} & 0 & \{ \eta \, : \, \eta(0)\neq \tt \} & v_\cc & v_\cc \\ 
    & & 1& \{ \eta \, : \, \eta(0)=\tt \} & w_\cc  & w_\cc - v_\cc \\
    \hline
    \gg & \{0\} & 0 & \{ \eta \, : \, \eta(0)\neq \aa \} & v_\gg & v_\gg\\ 
    & & 1 & \{ \eta \, : \, \eta(0)=\aa \} & w_\gg  & w_\gg - v_\gg \\
    \hline
           \hline
    \multicolumn{6}{|c|}{s^{(2)}= 1,\quad 
    \blambda^{(2)} = \sum_{a  \in \AA}(w_a-v_a), \quad \mbox{and}
      \quad 
      \blambda_0^{(2)} =\sum_{a \in \AA} v_a}\\
    \hline
  \end{array}
  \] 
Condition~\eqref{equa:ErgoFin2}   becomes \eqref{equa:RNergo}. 
  Note that since $s^{(2)}=1$, $w_a$ is not present in condition~\eqref{equa:RNergo},
  which  is thus weaker  than condition~\eqref{equa:ErgoFin-IV},
  while condition~\eqref{equa:Ergo2} becomes (recall \eqref{eq:K-m-s_1}--\eqref{eq:K-m-s_2})
\begin{equation} \label{equa:Ergo2bis}
  \min_{a \in \AA} v_a>0      \quad\mbox{and}\quad
  \max\{r_a^b, r_b^a \, : \, a  \in R, b \in Y
  \} < 4\min_{a \in \AA} v_a. 
\end{equation} 
\end{proof}
\section{Ergodicity through attractiveness 
}\label{sect:results-AttrErg}   
In this section we present an alternative approach to ergodicity,
based on the attractiveness of the process when it is present.  
For the sake of simplicity, we restrict ourselves to the finite
alphabet $\AA=\{\aa,\cc,\gg,\tt\}$.
While our first result (in Section \ref{sect:first-AttrErg}) is general,
the next ones (in Sections \ref{sec:ErgoAttrRN+YpR} and \ref{subsec:other-order}) 
depend on the specific dynamics of our generic  example,
the RN+YpR model with cut-and-paste  mechanism. The results stated in this section are 
proved in Section \ref{sect:ProofsAttrErg}. 

\subsection{Set-up and first result}\label{sect:first-AttrErg}  
We first recall the general set-up for attractiveness, relying on
\cite{ligg:IPS}. It requires  a total  order $\leq$ on $\AA$ which  induces a partial order
on $X$. Let us assume that such an order has been defined for the 4 elements of $\AA$,
that we write $1<2<3<4$ for the moment. 

Let $\eta$ and  $\xi$ be two configurations of $X$.  We say that $\eta
\leq \xi$ if for any $x\in\Z$, we have $\eta(x) \leq \xi(x)$.
We define $\MM$ as the class of bounded monotone functions $f$ on
$X$, that is, 
for all configurations $\eta$ and $\xi$ such that $\eta \leq \xi$, we have
$f(\eta) \leq f(\xi)$.  
The partial order on $X$ induces a stochastic order on   $\pp(X)$.   
For two  elements $\mu$ and $\nu$  of $\pp(X)$, 
we say that $\mu \preceq \nu$ if, for any $f \in \MM$, we
have $\mu(f) \leq \nu(f)$. 

According to Theorem~II.2.2 in \cite{ligg:IPS}, for any Feller process
on $X$ with semigroup $\{S(t) \, : \, t \geq 0\}$, the following two 
statements are equivalent. For any $\mu,\nu\in\pp(X)$,
\begin{itemize}
\item[(i)]
If $f \in \MM$, then $S(t)f \in \MM$, for any $t \geq 0$.
\item[(ii)]
If $\mu \preceq \nu$, then $\mu S(t) \leq \nu S(t)$, for any $t \geq 0$.
\end{itemize}
 A  Feller process  on $X$ with semigroup $\{S(t) \, : \, t \geq 0\}$
 is said  to be  \textit{attractive} if the equivalent conditions above are
satisfied. 

 An attractive process possesses  two special extremal invariant probability 
measures, the lower and the upper  one, 
that is for us $\mu_l=\lim_{t\to\infty}\delta_{\underline 1}S(t)$ 
and $\mu_u=\lim_{t\to\infty}\delta_{\underline 4}S(t)$, where $\delta_{\underline 1}$ (resp. $\delta_{\underline 4}$)
denotes the Dirac measure on the configuration $\eta$ such that $\eta(x)=1$ (resp. $\eta(x)=4$)
for all $x\in\Z$.  Note that they are translation invariant, since the dynamics is translation invariant. 
 They satisfy $\mu_l\preceq\mu_u$
and
any invariant probability measure $\mu$ is such that $\mu_l\preceq\mu\preceq\mu_u$. 
The process is ergodic if and only if $\mu_l=\mu_u$. 

To deal with our generic example, there are many possibilities 
to define an order on $\AA$. We have to choose one order that will
induce attractiveness of the dynamics.
We thus define the strict order $<$ on $\AA$ as
\begin{equation}\label{def:order-on-AA_ctag}
\cc < \tt < \aa < \gg
\end{equation}
that we use from now on.  Other possible choices will be detailed in 
Section \ref{subsec:other-order}.
\begin{prop}\label{prop:mu-a-t-c-g}
Assume that a Feller process on $X$ is attractive with respect to the order in 
\eqref{def:order-on-AA_ctag} and has translation invariant rates.
If any invariant and translation invariant probability measure $\mu$ satisfies
\begin{eqnarray}\label{eq:a=t}
\mu(\eta(0)=\aa)&=&\mu(\eta(0)=\tt),\\
\mu(\eta(0)=\cc)&=&\mu(\eta(0)=\gg),\label{eq:c=g}
\end{eqnarray}
then the process is ergodic, that is, $\mu_l=\mu_u$.
\end{prop}
\begin{rema}\label{rk:Y=R}
(i) Conditions \eqref{eq:a=t}--\eqref{eq:c=g} imply that $\mu(Y)=\mu(R) =1/2$.\\
(ii) Proposition \ref{prop:mu-a-t-c-g} can be extended to an alphabet of size 
$2n$ or $2n+1$ for an integer  $n\geq 1$ as follows.
Denoting $\AA=\{a_1,a_2,\cdots,a_k\}$ with the order  $a_1< a_2<\cdots< a_k$ 
and $k\in\{2n,2n+1\}$,
if we assume that any invariant and translation invariant probability measure 
 $\mu$  satisfies
\begin{eqnarray}\label{eq:ai=ak-i}
\mu(\eta(0)=a_i)&=&\mu(\eta(0)=a_{k-i})
\end{eqnarray}
for all $1\leq i\leq n$, then the process is ergodic. 
\end{rema}
\subsection{The attractive RN+YpR model with cut-and-paste  mechanism}\label{sec:ErgoAttrRN+YpR}
The next proposition will enable to concentrate on substitution models
to check attractiveness. 
\begin{prop}\label{prop:superp-attr}
Assume that the RN+YpR  model 
 is attractive. Then its superimposition  
with a cut-and-paste process is attractive as well.
\end{prop}
\begin{prop}\label{prop:RN+YpR-attr}
Assume that $\AA$ is endowed with the order 
\eqref{def:order-on-AA_ctag}. Under the conditions
\begin{eqnarray}
\label{eq:cond-rate-c}  0&=&r_\cc^\gg = r_\cc^\aa,\\
\label{eq:cond-rate-t}r_\tt^\aa &\leq& r_\tt^\gg, \\
\label{eq:cond-rate-a}r_\aa^\tt &\leq&  r_\aa^\cc,  \\
\label{eq:cond-rate-g}r_\gg^\cc = r_\gg^\tt &=& 0, 
\end{eqnarray} 
the RN+YpR model is attractive.
\end{prop}
\begin{rema} 
(i) The rates $r_a^b$ have to satisfy inequalities 
(see \eqref{eq:cond-rate-t}--\eqref{eq:cond-rate-a}), except
the ones corresponding to a transition to the biggest  or to the smallest element of
$\AA$ (with respect to the order 
\eqref{def:order-on-AA_ctag}), that have to be 0
(see \eqref{eq:cond-rate-c} and \eqref{eq:cond-rate-g}). \\
(ii) When the transition probability $p(.,.)$ is nearest-neighbor,
an application of Theorem 2.4 in \cite{borrello-1} gives that conditions \eqref{eq:cond-rate-c}--\eqref{eq:cond-rate-g}
are also necessary for attractiveness of the 
RN+YpR model with cut-and-paste mechanism.
\end{rema} 
 The next lemma is an immediate application of Proposition \ref{prop:RN+YpR-attr}.
\begin{lemm}\label{lem:attra-first-exa} 
 The T92+$\cc$p$\gg$ model, 
hence also the
JC+$\cc$p$\gg$ model, are attractive. 
\end{lemm} 
In view of our next results, we compute 
the first moments of any translation invariant and invariant 
probability measure for the RN+YpR model with cut-and-paste mechanism.
\begin{prop}\label{prop:RN+YpR_mu-a-t-c-g}
Let  $\mu\in\II\cap\SS$. 
It satisfies
\begin{eqnarray}\label{eq:RN+YpR_mu-a+g}
\mu(R)= \mu(\eta(0)\in\{\aa,\gg\} ) &=& \frac{v_{\textsc r}}{v_{\textsc y}+v_{\textsc r}},
\\\label{eq:RN+YpR_mu-c+t}
 \mu(Y)= \mu(\eta(0)\in\{\cc,\tt\} ) &=& \frac{v_{\textsc y}}{v_{\textsc y}+v_{\textsc r}},
\end{eqnarray}
and 
\begin{eqnarray}\label{eq:RN+YpR_mu-a}
\mu(\eta(0)=\aa) &=& 
               \frac{v_\aa\, \mu(Y)+w_\aa\,\mu(R)-r_{\textsc r}}{w_{\textsc r}+v_{\textsc y}}, \\ \label{eq:RN+YpR_mu-c}
\mu(\eta(0)=\cc) &=& 
         \frac{v_\cc\,\mu(R)+ w_\cc\, \mu(Y)-r_{\textsc y}}{w_{\textsc y}+v_{\textsc r}},  \\ \label{eq:RN+YpR_mu-g}
\mu(\eta(0)=\gg) &=& 
           \frac{v_\gg\, \mu(Y)+w_\gg\,\mu(R)+r_{\textsc r}}{w_{\textsc r}+v_{\textsc y}}, \\ \label{eq:RN+YpR_mu-t}
\mu(\eta(0)=\tt) &=& 
                \frac{v_\tt\,\mu(R)+w_\tt\, \mu(Y)+r_{\textsc y}}{w_{\textsc y}+v_{\textsc r}}, 
\end{eqnarray}
where we abbreviated
\begin{eqnarray} \label{eq:abbrev-v1-v2}
 v_{\textsc y}  = v_\tt+v_\cc \quad &;& \quad v_{\textsc r}  = v_\aa+v_\gg, \\ \label{eq:abbrev-w1-w2}
 w_{\textsc y}  = w_\tt+w_\cc \quad &;&  \quad w_{\textsc r}  = w_\aa+w_\gg, 
\end{eqnarray}
and
\begin{eqnarray} \nonumber
-r_{\textsc y} = -r_{\textsc y}(\mu)  &=& 
\sum_{a\in R}  r_\cc^a \mu\left((\eta(0),\eta(1))=(\tt,a)\right)
\\&&\quad -  \sum_{a\in R}
 r_\tt^a  \mu\left((\eta(0),\eta(1))=(\cc,a)\right), \label{eq:abbrev-r1}\\
 \nonumber
-r_{\textsc r}= -r_{\textsc r}(\mu)  &=& 
\sum_{a\in Y} r_\aa^a \mu\left((\eta(-1),\eta(0))=(a,\gg)\right)\\
&&\quad -  \sum_{a\in Y} r_\gg^a  \mu\left((\eta(-1),\eta(0))=(a,\aa)\right).\label{eq:abbrev-r2}
\end{eqnarray}
\end{prop}
\begin{rema}\label{rk:bgp}
The values \eqref{eq:RN+YpR_mu-a+g}--\eqref{eq:RN+YpR_mu-t}
were already computed in  \cite[Proposition 14]{piau:solv}: 
indeed since we consider first moments of 
a translation invariant probability measure
for a translation invariant dynamics, the cut-and-paste mechanism 
disappears from our computations.
However we cannot obtain two-points correlations as in  
\cite[Proposition 15]{piau:solv}, since there
the cut-and-paste mechanism is present in computations. 
\end{rema}
We can now prove part (i) of Theorem \ref{th:ergo-first-exa}; the first part of 
(ii) will be proved later on, thanks to Proposition \ref{prop:nu-diag}.
\begin{proof} \textit{(of Theorem \ref{th:ergo-first-exa}, (i)).}
This is an application of Proposition \ref{prop:mu-a-t-c-g}: 
We have to check  that equalities \eqref{eq:a=t}--\eqref{eq:c=g}
are satisfied for both examples. Since the JC+$\cc$p$\gg$ model 
is a particular case of the
T92+$\cc$p$\gg$ model, it is enough to consider the latter.
By Lemma \ref{lem:attra-first-exa} and Proposition \ref{prop:superp-attr},
the T92+$\cc$p$\gg$ model with cut-and-paste mechanism is attractive.
 Then we compute,  using that $\mu$ is translation invariant, 
\begin{eqnarray*}\label{eq:T92+CpG_v1v2}
v_{\textsc y} = v_{\textsc r}=v &;& w_{\textsc y} = w_{\textsc r}=w,\\\label{eq:T92+CpG_ACGT}
\mu(Y)=\mu(R)&= & 1/2,
\\\label{eq:T92+CpG_r1r2}
r_{\textsc y} = r \mu ((\eta(0),\eta(1))=(\cc,\gg))&= & -r_{\textsc r}= r \mu ((\eta(-1),\eta(0))=(\cc,\gg)),\\\label{eq:T92+CpG_mu-a-t}
\mu(\eta(0)=\aa) &=& 
\frac{1-\theta}{2}+\frac{r_{\textsc y}}{v+w}=\mu(\eta(0)=\tt),\\ \label{eq:T92+CpG_mu-c-g}
\mu(\eta(0)=\cc) &=& 
\frac{\theta}{2}+\frac{r_{\textsc r} }{v+w}=\mu(\eta(0)=\gg). 
\end{eqnarray*}
Thus  \eqref{eq:a=t}--\eqref{eq:c=g}
are satisfied for the
T92+$\cc$p$\gg$ model with cut-and-paste mechanism. 
\end{proof}  
Thanks to Proposition \ref{prop:RN+YpR_mu-a-t-c-g}, it is possible 
to relax the assumptions of 
Proposition \ref{prop:mu-a-t-c-g} to derive ergodicity:
 \begin{prop}\label{prop:mu-a-t_or_c-g}
Assume that the RN+YpR model with cut-and-paste mechanism is attractive 
with respect to the order in 
\eqref{def:order-on-AA_ctag}. 
If  $\mu_l$ and  $\mu_u$ satisfy either \eqref{eq:a=t} or \eqref{eq:c=g},
then the process is ergodic.
\end{prop}
\begin{rema}\label{rk:Ynot=R}
Under the assumption of Proposition \ref{prop:mu-a-t_or_c-g}, 
we may have  $\mu(Y)\not=\mu(R)$,
where $\II=\{\mu\}$.
\end{rema}
 Another way to look for ergodicity is to investigate the 
monotone coupling measure $\nu$ of two
 ordered probability measures $\mu_1$ and $\mu_2$. 
The probability measure $\nu$ on $X \times X$
is a \textit{monotone coupling measure} of $\mu_1$ and $\mu_2$, if its
 marginals are $\mu_1$ and $\mu_2$, and it satisfies
\begin{equation}\label{eq:nu-strassen}
\nu((\eta,\xi):\eta\leq \xi) =1.
\end{equation}
 Such a coupling exists by Strassen's Theorem since 
$\mu_1\preceq\mu_2$  (see \cite[Theorem II.2.4]{ligg:IPS}).
\begin{prop}\label{prop:nu-some-coupl-0}
Assume that $\AA$ is endowed with the order \eqref{def:order-on-AA_ctag}.
Let $\nu$ be a monotone coupling measure  of two
 translation invariant probability measures 
 $\mu_1$ and $\mu_2$ such that $\mu_1\preceq\mu_2$, and which both satisfy 
 \eqref{eq:RN+YpR_mu-a+g}--\eqref{eq:RN+YpR_mu-c+t}.
Then we have
\begin{equation}\label{eq:nu-some-coupl-0}
\nu ( (\eta(0),\xi(0))\in \{(\cc,\gg),(\tt,\gg),(\cc,\aa),(\tt,\aa)\} )=0.
\end{equation}
\end{prop}
This proposition applies to the coupling measure $\nu$ of the lower
and upper invariant probability measures $\mu_l$ and $\mu_u$,
when the dynamics is attractive  (with respect to the order in 
\eqref{def:order-on-AA_ctag}). 
Thus, proving that $\nu ( (\eta(0),\xi(0))\in \{(\aa,\gg),(\cc,\tt)\} )=0$ 
would imply ergodicity.
\begin{prop}\label{prop:nu-diag}
Let $\nu$ be a monotone coupling measure of 
$\mu_l$ and $\mu_u$  
when the process is attractive  with respect to the order in 
\eqref{def:order-on-AA_ctag}.  Assume that the rates satisfy 
one of the 3 following conditions, 
\begin{itemize}
\item[(a)] $r_\aa^\cc = r_\aa^\tt$; 
\item[(b)] $r_\tt^\gg = r_\tt^\aa$;
\item [(c)] $\Big((\alpha)\mbox{ and }(\gamma)\Big)$ or $(\beta)$ or $(\delta)$;
\end{itemize}
where 
\begin{eqnarray*}\label{eq:alpha}
(\alpha)&& (r_\tt^\gg - r_\tt^\aa)- r_\aa^\cc  \leq  0, \\ \label{eq:beta}
(\beta)&& 0<(r_\tt^\gg - r_\tt^\aa)- r_\aa^\cc\leq v_\tt + v_\cc+w_\gg + w_\aa, \\ \label{eq:gamma}
(\gamma)&& (r_\aa^\cc - r_\aa^\tt ) - r^\gg_\tt  \leq 0, \\ \label{eq:delta}
(\delta)&& 0< (r_\aa^\cc - r_\aa^\tt ) - r^\gg_\tt   \leq w_\tt + w_\cc+v_\gg + v_\aa.
\end{eqnarray*}
 Then we have
\begin{equation}\label{eq:for-erg-by-nu}
\nu ( (\eta(0),\xi(0))\in \{(\aa,\gg),(\cc,\tt)\} )=0,
\end{equation}
hence the RN+YpR model with cut-and-paste mechanism is ergodic.
\end{prop}
\begin{rema}\label{rk:again-examples-ok}
This proposition gives another proof of Theorem \ref{th:ergo-first-exa}, (i), 
since  the T92+$\cc$p$\gg$ model, hence the JC+$\cc$p$\gg$ model, both 
with cut-and-paste mechanism, satisfy also 
$(\alpha)$ and $(\gamma)$ in the set (c) of conditions of
Proposition \ref{prop:nu-diag}.
\end{rema}
\subsection{Other order relations  for the  RN+YpR model with cut-and-paste  mechanism}\label{subsec:other-order}
 We chose the order \eqref{def:order-on-AA_ctag} on $\AA$, that gave  relations on the rates
$r_a^b$ for attractiveness, and eventually ergodicity. 
In Theorem \ref{th:ergo-first-exa} we saw that those relations gave attractiveness and 
ergodicity for the  JC+$\cc$p$\gg$, T92+$\cc$p$\gg$, and some RNc+YpR models with cut-and-paste mechanism.
Are there other possible orders on $\AA$ that would give attractiveness
of the RN+YpR model?
There are a priori 24 possibilities.

In what follows, we refer to the proof of Proposition \ref{prop:RN+YpR-attr}, done in Section \ref{sect:ProofsAttrErg}. 
There, we detail the coupling transitions starting from two ordered 
configurations (with respect to the order \eqref{def:order-on-AA_ctag}), 
and forbid transitions that would break this order between the configurations. 
Going  to the coupling tables in this proof, we see that we cannot take an order relation that
would `separate' the values in $Y$ and $R$ : Let us try for instance $\cc<\aa<\tt<\gg$; 
then we cannot forbid the transition from $(\cc,\aa)$ to $(\tt,\aa)$.
This fact forbids 16 possibilities of order. 

Then, once we do not separate the values in $Y$ and $R$, we are left 
with the following 8 possibilities, 
%
\begin{eqnarray*}
(O1) && \cc < \tt < \aa < \gg, \\ 
(O2) && \gg < \aa < \tt < \cc, \\ 
(O3) && \tt < \cc < \aa < \gg, \\ 
(O4) && \tt < \cc < \gg < \aa, \\ 
(O5) && \cc < \tt < \gg < \aa, \\ 
(O6) && \aa < \gg < \cc < \tt, \\ 
(O7) && \aa < \gg < \tt < \cc, \\ 
(O8) && \gg < \aa < \cc < \tt,
\end{eqnarray*} 
with the attractiveness conditions they induce, 
by proceeding as 
in Proposition \ref{prop:RN+YpR-attr} and doing the ad-hoc permutations. 
Indeed, there we worked with the order 
\eqref{def:order-on-AA_ctag}, that we now denote as
order \textit{(O1)}; if we write it as $1<2<3<4$, the attractiveness 
conditions \eqref{eq:cond-rate-c}--\eqref{eq:cond-rate-g}
are written 
\begin{eqnarray}
\label{eq:other-type_cond-rate-c}0&=&r_1^3 = r_1^4, \\
\label{eq:other-type_cond-rate-t}r_2^3 &\leq& r_2^4,\\
\label{eq:other-type_cond-rate-a}r_3^1 &\geq& r_3^2, \\
\label{eq:other-type_cond-rate-g}r_4^1 = r_4^2&=&0,
\end{eqnarray} 
 and the conditions in Proposition \ref{prop:nu-diag} become \\
\begin{equation}\label{eq:other-conditions} (\widetilde a) \,\, r_3^1 = r_3^2 \,; \quad
(\widetilde b) \,\, r_2^4 = r_2^3\,; \quad
(\widetilde c) \,\,\Big((\widetilde\alpha)\,\, and \,\,(\widetilde\gamma)\Big)\,\, 
or \,\,(\widetilde\beta)\,\, or \,\,(\widetilde\delta); 
\end{equation}
where 
\begin{eqnarray}\label{eq:other-alpha}
(\widetilde\alpha)&& (r_2^4 - r_2^3)- r_3^1  \leq  0, \\ \label{eq:other-beta}
(\widetilde\beta)&& 0<(r_2^4 - r_2^3)- r_3^1\leq v_2 + v_1+w_4 + w_3, \\ \label{eq:other-gamma}
(\widetilde\gamma)&& (r_3^1 - r_3^2 ) - r^4_3  \leq 0, \\ \label{eq:other-delta}
(\widetilde\delta)&& 0< (r_3^1 - r_3^2 ) - r^4_3   \leq w_2 + w_1+v_4 + v_3.
\end{eqnarray} 
If we take also into account the constraint that
we want to keep the result for the JC+$\cc$p$\gg$ and T92+$\cc$p$\gg$ 
models with cut-and-paste mechanism, then,
 among the 7 possibilities after \textit{(O1)}, only one is possible, 
  which is \textit{(O2)}, that is,
 $\gg<\aa<\tt<\cc$. \\
 But the other orders enable to deal with other dynamics, for instance the 
RNc+YpR model with cut-and-paste mechanism, for which
 we can now derive attractiveness conditions then prove Theorem \ref{th:ergo-first-exa}, (ii). 
\begin{lemm}\label{lem:attra-RNc+YpR} 
 The RNc+YpR model is attractive if its rates satisfy either \eqref{eq:attra-rates-RNc+YpR-1}
 or \eqref{eq:attra-rates-RNc+YpR-2}.
\end{lemm} 
\begin{proof}
We have to check conditions \eqref{eq:other-type_cond-rate-c}--\eqref{eq:other-type_cond-rate-g}
respectively for the orders $(O1)$ to $(O8)$: The orders $(O1)$ and $(O2)$ yield 
\eqref{eq:attra-rates-RNc+YpR-1}, the orders $(O4)$ and $(O6)$ yield 
\eqref{eq:attra-rates-RNc+YpR-2}, while the other orders yield a trivial case, where all the rates are equal to 0.
\end{proof}
\begin{proof} \textit{(of Theorem \ref{th:ergo-first-exa}, (ii)).}
The two possible sets of rates of the attractive RNc+YpR   model
with cut-and-paste mechanism (that is \eqref{eq:attra-rates-RNc+YpR-1}
 or \eqref{eq:attra-rates-RNc+YpR-2}, see Lemma \ref{lem:attra-RNc+YpR}) 
satisfy $(\widetilde\alpha)$ and $(\widetilde\gamma)$ in the set $(\widetilde c)$ of the above conditions \eqref{eq:other-conditions}, for the respective orders $(O1)$ and $(O4)$.
\end{proof} 
%
%
%
\section{Proofs through generalized duality} \label{sect:ProofsDuality}
 To prove Theorems~\ref{theo:Ergo} and \ref{theo:Ergo2} we proceed as follows: 
In Section \ref{sect:Graph},
we provide a graphical construction of the process, which yields a generalized dual of the process.
 Then this dual is dominated by a branching process, for which we derive a condition for extinction.
 This one implies exponential ergodicity of the process; see Section \ref{sect:ProofThErgoDual}. \\
 Although our proofs are quite similar in spirit to those of \cite{ferrari:ErgoSpinStirring},
 we chose to give details for the sake of completeness, and to highlight the places were
 they are different.  
\subsection{Graphical constructions and dual process}\label{sect:Graph} 
We  adapt to
our  context the graphical constructions of  \cite{ferrari:ErgoSpinStirring}. 
We start in section \ref{sect:SubsGraphRepr} with the  substitution  process
in our  two different cases: either  the minimal 
substitution rate is  positive,  or the  pregenerator $\LL_1$  is
decomposed  in  a  sum   of  pregenerators.  
Then, in Section \ref{sect:TransGraphRepr} we   provide  the   graphical  construction   of  the
cut-and-paste process. Finally, in Section \ref{sect:GeneDualProc}  we construct a (generalized) dual
process  (that  is,  a  non-Markovian dynamics) for  the  process  with
pregenerator $\LL$.  
\subsubsection{Graphical  construction of  the process  with pregenerator
  $\LL_1$} \label{sect:SubsGraphRepr}
\noindent$\bullet$ \textbf{{Under Assumption~\eqref{equa:Mpositif}:  the  minimal
   substitution rate  is positive}} 

Recall the decomposition  \eqref{equa:SpinFlipRate} 
of the rate function $c(\cdot, \cdot)$ introduced in
Section~\ref{sect:first-ergo_duality}  and the notations there.
 For $x\in\Z,a\in\AA,j\in J(a)$,   let 
\begin{equation}\label{eq:tauAj}
A_j(x,a) = \tau_x^{-1}A_j(a)
\end{equation} 
 i.e. $\eta \in A_j(x,a)$ if and only if $\tau_x\eta \in A_j(a)$.
First injecting \eqref{equa:SpinFlipRate} in \eqref{equa:genesubs1}, 
then using \eqref{equa:partition} yields a
rewriting of the pregenerator $\LL_1$ as  
\begin{eqnarray} \nonumber 
\LL_1 f (\eta)  
&=&
 \sum_{x\in \Z}  \sum_{a \in \AA}  \sum_{j \in  J(a)\setminus\{0\}} 
  \blambda_j(a) \sum_{\ell 
    \geq j} \1_{\{ \eta \in A_\ell(x,a)\}}\left[f(\eta^x_a) - f(\eta)\right] \\ \nonumber 
&& +  \sum_{x\in \Z}  \sum_{a \in \AA}  \blambda_0(a) \sum_{\ell \geq 0} \1_{\{ \eta \in
  A_\ell(x,a)\}}\left[f(\eta^x_a) - f(\eta)\right] \\ \label{equa:Multinomial}
& = & \LL_1^{b,1} f (\eta) + \LL_1^{n,1} f (\eta)
\end{eqnarray}
 where 
\begin{eqnarray}\label{equa:gene-birth1} 
\LL_1^{b,1} f (\eta)& = & \sum_{x\in  \Z} \sum_{a \in \AA} \sum_{j \in J(a)\setminus\{0\}}
 \blambda_j(a) \sum_{\ell 
    \geq j} \1_{\{ \eta \in A_\ell(x,a)\}}\big[f(\eta^x_a) - f(\eta)\big]  \\ \label{equa:gene-noise1} 
\LL_1^{n,1} f (\eta)& = &   \sum_{x \in  \Z}    \blambda_0  \sum_{a \in  \AA}
\frac{  \blambda_0(a)}{ \blambda_0} \left[  f(\eta^{x}_a) -
  f(\eta)\right].  
\end{eqnarray}
For the branching process, $\LL_1^{b,1}$ will induce the births, while $\LL_1^{n,1}$
corresponds to a ``noise'' part that will induce the deaths. Note that, contrary to 
\cite{ferrari:ErgoSpinStirring}, our rates for this noise part are not uniform (this is why
we will get a better r.h.s. in \eqref{equa:ErgoFin} than in \eqref{equa:Ergo}, 
as it will be explained in the proof of Remark
\ref{rema:assu-1}).

We now define, on the graphical representation, 
marks $(j,a)$ and $(\delta,a)$  (for $a\in\AA,j \in J(a)\setminus \{0\}$) 
that will induce respectively births and deaths
in the branching process. 
 The corresponding families of random variables are all mutually independent. \\ 
Let ${\mathfrak M} = \{ (M_u(x,a), u \geq 0)
\, : \, x \in \Z, a \in \AA \}$  be 
a family of independent Poisson point processes (PPPs) such that the
rate   of  $(M_u(x,a), u \geq 0)$  
is $\lambda (a)$ defined in~\eqref{equa:Lambda}.\\
Let $\UU = \{ (U_n(x,a),n \geq 0)
\, : \, x \in \Z, a \in \AA \}$ be 
a family of independent random variables, all uniformly distributed 
on~$[0,1]$. The $n$-th occurrence of $(M_u(x,a), u \geq 0)$
is  marked  $(j,a)$ with $j \in J(a)\setminus \{0\}$ if 
\[
  \frac{   \lambda_{j-1}(a)  -   \lambda_0(a)   }{\lambda(a)}  <
  U_n(x,a) < 
  \frac{  \lambda_{j}(a)- \lambda_0(a)}{  \lambda(a)}. 
\]  
 Thus the  $(j,a)$  marks  
 are distributed according to a PPP with rate $\blambda_{j}(a) $. 

Let ${\mathfrak M}^0 = \{ (M^0_u(x), u \geq 0) \,  : \, x \in \Z\}$ be
a family of independent PPPs such that the
rate   of  $(M^0_u(x), u \geq 0)$ 
  is  $\blambda_0$  defined by~\eqref{equa:bLambda0}.  
  Let  $\UU^0 =  \{ (U^0_n(x),n\geq 0) \, : \, x \in \Z\}$ 
  be a family of independent discrete random variables with values
in $\AA$ such that 
\begin{equation}\label{eq:U0}
\PP(U^0_n(x)=a) = \frac{\lambda_0(a)}{\overline \lambda_0}.
\end{equation}
  The $n$-th  occurrence  of   $(M^0_u(x), u  \geq  0)$  is  marked
$(\delta,a)$ if $U^0_n(x)=a$.  

 The evolution of the process $(\eta_t)_{t\geq 0}$ is now determined by
 this graphical representation as follows. 
Let $\omega \in \Omega$ be a configuration of the marked PPPs.  Fix a site
$x\in  \Z$ and  a time  $t >  0$, and  let $0<T_1  \leq  \dots \leq
T_{d-1} < t$ be the times  of the successive marks present at site $x$
in  the time  interval $[0,t]$.   Set  $T_0=0$ and  $T_d=t$.  Then  we
define $\eta_s(x)=\eta_{T_i}(x)$, for   $s \in [T_i,T_{i+1}),i<d$,
 where $\eta_{T_i}(x)$ is constructed with the following recipe.

Suppose that  the configuration at time $T_i^-$  is $\eta_{T_i^-}$. By
definition of $T_i$, a mark of $\omega$ is present at site $x$
at time~$T_i$. There are two possibilities:
\begin{itemize} 
\item[1(a).]
A   $(j,a)$-mark  with   $j  \in   J(a)\setminus  \{0\}$.   If  the
configuration $\eta_{T_i^-}$ belongs to at least 
one of the sets $A_\ell(x,a)$  with $\ell \geq j$, then substitute the
letter at  $x$ by  $a$, so that  $\eta_{T_i}=(\eta_{T_i^-})^x_a$. Otherwise,
nothing happens.
\item[1(b).]
 A $(\delta,a)$-mark. Substitute the
letter at $x$ by $a$, so that $\eta_{T_i}=(\eta_{T_i^-})^x_a$.
 Note that this substitution is independent of $\eta_{T_i^-}$, 
whence the term ``noise'' for $\LL_1^{n,1}$. 
\end{itemize}
The  fact that  this recipe  produces the  desired  substitution rates
$c(\cdot, \cdot)$  comes  from   the  rewriting~\eqref{equa:Multinomial}
of~\eqref{equa:genesubs1},  and  the   thinning  property  of  Poisson
processes. 
Moreover a percolation argument (see e.g. \cite{durrett:TenLecturesIPS}, Section 2) 
implies  that only a 
finite (random) number of sites influence the evolution of a fixed site, 
hence the previous description yields a well-defined 
dynamics.

\noindent$\bullet$ \textbf{Under  Assumptions~\eqref{equa:genesubs4}
  and~\eqref{equa:Mpositif2}: decomposition of the pregenerator as a sum}
  
Recall  the  decomposition  of  the  rate  functions  $c^{(i)}(\cdot, \cdot)$
introduced   in    Section~\ref{sect:second-ergo_duality}   and   the  
notations there. Proceeding as for \eqref{equa:Multinomial}  
gives a rewriting of the pregenerators $\LL_1^{(i)}$  and $\LL_1$ as 
\begin{eqnarray} 
  \nonumber 
  \LL_1^{(i)} f (\eta) 
  &=&   \sum_{x\in   \Z}   \sum_{a  \in   \AA}   \sum_{j   \in
      J^{(i)}(a)\setminus\{0\}}  
    \blambda_j^{(i)}(a)   \sum_{\ell   \geq    j}   \1_{\{\eta\in
    A_\ell^{(i)}(x,a)\}}\big[f(\eta^x_a) - f(\eta)\big] \\  \nonumber   
&& + \sum_{x\in   \Z}   \sum_{a  \in   \AA}
     \blambda_0^{(i)}(a)  \sum_{\ell \geq  0}
    \1_{\{  \eta \in A_\ell^{(i)}(x,a)\}}
    \big[f(\eta^x_a)  - f(\eta)\big] \\  \label{equa:Multinomial-i} 
  \LL_1  f (\eta)  
&=& \LL_1^{b,2} f (\eta) + \LL_1^{n,2} f (\eta)
\end{eqnarray}
 where 
\begin{eqnarray}\label{equa:gene-birth2} 
\LL_1^{b,2} f (\eta)& = & \sum_{i=1}^d  \sum_{x\in \Z}  \sum_{a  \in \AA}
  \sum_{j \in J^{(i)}(a)\setminus\{0\}} 
  \blambda_j^{(i)}(a)  \sum_{\ell  \geq   j}  \1_{\{   \eta  \in
  A_\ell^{(i)}(x,a)\}}\big[f(\eta^x_a) -   f(\eta)\big] \\ \label{equa:gene-noise2} 
\LL_1^{n,2} f (\eta)& = &   
  \sum_{x \in  \Z}   \blambda_{0,d}  \sum_{a \in  \AA}
  \frac{\blambda_{0,d}(a)}  {\blambda_{0,d}}  \left[ f(\eta^{x}_a)
    -     f(\eta) \right].
\end{eqnarray}
 and, as in \eqref{eq:tauAj}, for $x\in\Z,a\in\AA,j\in J^{(i)}(a)$,  
 $\eta \in A_j^{(i)}(x,a)$ if and only if $\tau_x\eta \in A_j^{(i)}(a)$. 
 
Let ${\mathfrak M}^{(i)} = \{ (M_u^{(i)}(x,a), u \geq 0)
\, : \, x \in \Z, a \in \AA \}$ be 
a family of independent PPPs such that the
rate   of  $(M_u^{(i)}(x,a), u \geq 0)$ 
is  $\lambda^{(i)} (a)$ (see \eqref{equa:Lambda-i-total}). 
 Let $\UU^{(i)} = \{ (U_n^{(i)}(x,a),n \geq 0)
 \, : \, x \in \Z, a \in \AA \}$ be 
 a family of independent random variables, all with uniform law  on~$[0,1]$.
 The 
 $n$-th occurrence of $(M_u^{(i)}(x,a), u \geq 0)$ 
  is  marked  $(j,a,i)$ with $j \in J^{(i)}(a)\setminus \{0\}$ if 
\[
  \frac{ \lambda_{j-1}^{(i)}(a) -
    \lambda_0^{(i)}(a) }{\lambda^{(i)}(a)} < 
  U_n^{(i)}(a) < 
  \frac{\lambda_{j}^{(i)}(a)-\lambda_0^{(i)}(a)}{ \lambda^{(i)}(a)}.  
\]  
Thus  the  $(j,a,i)$  marks  
are distributed according to a PPP with rate $\blambda_{j}^{(i)}(a) $. 

Let ${\mathfrak M}^{0,d} = \{ (M^{0,d}_u(x), u \geq 0) \,  : \, x \in \Z\}$ be
a family of independent PPPs such that the
rate   of  $(M^{0,d}_u(x), u \geq 0)$ 
  is  $ \blambda_{0,d}$  defined  by~\eqref{equa:branch_inf_i}.  Let
  $\UU^{0,d} = \{ (U^{0,d}_n(x),n \geq 0) \, : \, x \in 
\Z\}$ be a family of independent discrete random variables with values
in $\AA$ such that 
\begin{equation}\label{eq:U0-i}
\PP(U^{0,d}_n(x)=a) = \frac{\blambda_{0,d}(a)}{\overline \lambda_{0,d}}.
\end{equation}
  The $n$-th  occurrence  of  $(M^{0,d}_u(x), u  \geq  0)$  is  marked
$(\delta,a)$ if $U^{0,d}_n(x)=a$.  

Let $\omega \in \Omega$ be a configuration of the marked PPPs.  Fix a site
$x\in  \Z$ and  a time  $t >  0$, and  let $0<T_1  \leq  \dots \leq
T_{d-1} < t$ be the times  of the successive marks present at site $x$
in  the time  interval $[0,t]$.   Set  $T_0=0$ and  $T_d=t$.  Then  we
define  $\eta_s(x)=\eta_{T_i}(x)$,  for  $s \in  [T_i,T_{i+1})$  where
$\eta_{T_i}(x)$ is constructed with the following recipe.

Suppose that  the configuration at time $T_i^-$  is $\eta_{T_i^-}$. 
By definition of $T_i$, a mark of $\omega$ is present at site $x$ at time~$T_i$. 
There are two possibilities:
\begin{itemize} 
\item[1(a).]
A $(j,a,i)$-mark with $j \in J^{(i)}(a)\setminus \{0\}$.  If the 
configuration $\eta_{T_i^-}$ belongs to at least 
one of the sets $A_\ell^{(i)}(x,a)$  with $\ell \geq j$, then substitute the
letter at  $x$ by  $a$, so that  $\eta_{T_i}=(\eta_{T_i^-})^x_a$. Otherwise,
nothing happens.
\item[1(b).]
 A $(\delta,a)$-mark. Substitute the
letter at $x$ by $a$, so that $\eta_{T_i}=(\eta_{T_i^-})^x_a$.
\end{itemize}
%
%
\subsubsection{Graphical  construction of  the process  with pregenerator
  $\LL_2$} \label{sect:TransGraphRepr}
Once again, we  use a Harris graphical construction  based on a family
of independent  PPPs  indexed by $\Z  \times \Z$; 
it is adapted from~\cite{andjelAl:LLNExclusion}  (which deals with an exclusion process).  
 A Borel-Cantelli argument shows that  only a finite number of
Poisson processes are involved in  the computation of the evolution of
a site $x$ until a fixed time~$t$. 

Let $\NN=\{(N_u(x,y), u \geq 0) \, : \, (x,y) \in \Z \}$ be a family
of independent  PPPs  such that the  rate of  the process 
indexed by $(x,y)$ is $p(x,y)$. At each of its arrival times 
 and each site $z$ such that $x \leq
z  \leq  y $  (or  $y  \leq z  \leq  x  $ if  $x>y$),  we  put a  mark
$(\circlearrowright,x,y)$.  

Let $\omega \in \Omega$ be a configuration of the marked PPP.  Fix a site
$ z\in \Z$ and a time $t > 0$, and let $0<T_1 \leq \dots \leq
T_{d-1} < t$ be the times  of the successive marks present at site $z$
in  the time  interval $[0,t]$.   Set  $T_0=0$ and  $T_d=t$.  Then  we
define  $\eta_u(x)=\eta_{T_i}(x)$,  for  $u \in  [T_i,T_{i+1})$  where
$\eta_{T_i}(x)$ is constructed with the following recipe.

Suppose that  the configuration at time $T_i^-$  is $\eta_{T_i^-}$. By
definition of  $T_i$, a mark $(\circlearrowright,x,y)$  of $\omega$ is
present at site $z$ at time~$T_i$. There are two possibilities: 
\begin{itemize}
\item[2(a).] $ x<y$. In this case, the contents of 
  sites $x$, $x+1$, $\dots$, $y$ are right circularly permuted so that
  $\eta_{T_i} = \sigma_{x,y}(\eta_{T_i^-})$. 
\item[2(b).] $ x>y$. In this case, the contents of 
  sites $y$, $y+1$, $\dots$, $x$  are left circularly permuted so that
  $\eta_{T_i} = \sigma_{x,y}(\eta_{T_i^-})$.  
\end{itemize}
 This recipe  produces the  desired  cut-and-paste rates
given by $p(.,.)$. 

On Figure~\ref{figu:EvolTrans}, given  a configuration $\omega$  of marked
PPPs, one can see the evolution of sites $1$ to $5$.
\begin{figure}[ht!]
  \centering
  \begin{minipage}[b]{0.47\textwidth}
    \centering
    \includegraphics
    [viewport=0             0            265            200,clip,width=\textwidth]
    {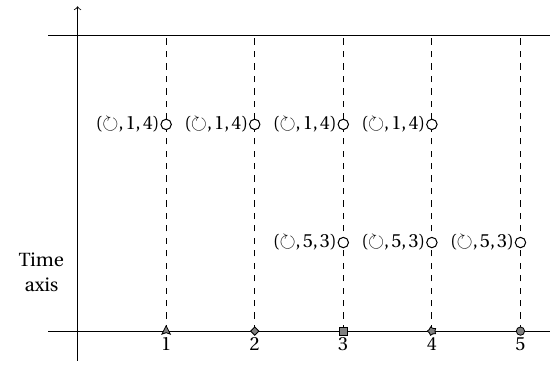} 
  \end{minipage}
  \hspace{0.375cm}
  \begin{minipage}[b]{0.47\textwidth}
    \centering
    \includegraphics
    [viewport=0             0            265            200,clip,width=\textwidth]
    {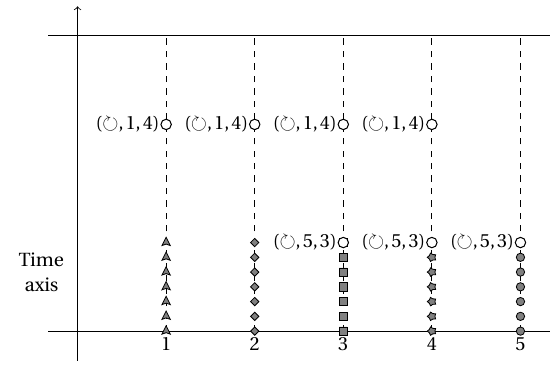} 
  \end{minipage} \\
  \vspace{0.2cm}
  \begin{minipage}[b]{0.47\textwidth}
    \centering
    \includegraphics
    [viewport=0             0            265            200,clip,width=\textwidth]
    {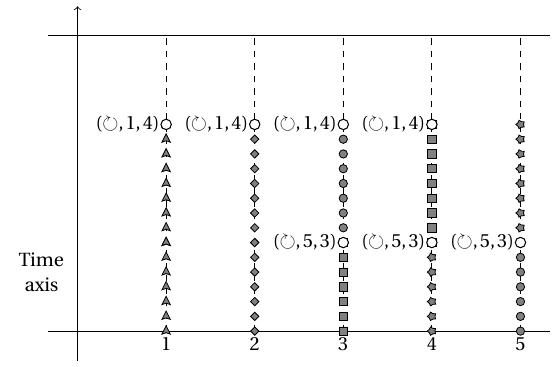} 
  \end{minipage}
  \hspace{0.375cm}
  \begin{minipage}[b]{0.47\textwidth}
    \centering
    \includegraphics
    [viewport=0             0            265            200,clip,width=\textwidth]
    {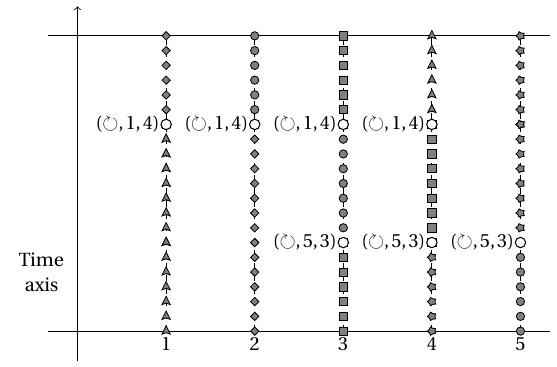} 
  \end{minipage}
  \caption{
    Evolution  of  the  sites  $1$ to  $5$  under  pregenerator
    $\LL_2$ given a realization of the marked PPPs. It should be read
    by line from the  left to the right. Here,
{\protect\tikz \protect\draw [decorate with=rectangle, paint=black] (0,0) to (0,0.1);} stands for $\aa$, 
{\protect\tikz \protect\draw [decorate with=dart, paint=black] (0,0) to (0,0.1);} for $\tt$, 
{\protect\tikz \protect\draw [decorate with=star, paint=black] (0,0) to (0,0.1);} for $\cc$ and 
{\protect\tikz \protect\draw [decorate with=diamond, paint=black] (0,0) to (0,0.1);} for $\gg$.    
    } \label{figu:EvolTrans} 
\end{figure}
\subsubsection{Construction of the dual process} \label{sect:GeneDualProc} 
Thanks     to    the     constructions    provided     in    Sections 
\ref{sect:SubsGraphRepr} and \ref{sect:TransGraphRepr},   
 to  have a  graphical  representation  of  the process  with
pregenerator  $\LL$, it  suffices  to  multiply the  rates  of $\NN$  by
$\rho$,    and to    assume    that   ${\mathfrak M}$,    ${\mathfrak M}^0$,
$\UU$, $\UU^0$ and 
$\NN$, are  mutually independent.

Now, we  turn to the construction  of the generalized  dual process of
$\LL$. 
It is a marked  branching structure
constructed on the space $\Z \times [0,+\infty)$. 

Fix a finite set of sites $D  \subset \Z$ and a time $t$. Suppose that
for the time  interval $[0,t]$, we have a  realization $\omega$ of the
marked PPPs  described above.  We reverse the  time direction calling
$\widehat u =  t-u$  and  we  construct a  space-time branching  structure
contained in $\Z  \times [\widehat 0, \widehat t]$ with  base $(D,\widehat 0)$ and
top  $(D_{\widehat  t},  \widehat  t)$.   We proceed  by  induction  with  the
following recipe.

Suppose that the spatial projection of the structure at
time $\widehat u$ is $D_{\widehat u}$.  Let $\widehat T$ be the first Poisson mark
after  $\widehat u$ involving  some site  of $D_{\widehat u}$. There  are the
following possibilities. 
\begin{itemize}
\item[1(a).] 
A  $(j,a)$-mark  involving site  $x  \in  D_{\widehat  T^-}$ with  $j  \in
J(a)\setminus\{0\}$. In this case,  the point $(x,\widehat T)$ is marked
$(j,a)$ and the set $D_{\widehat T}$ will be  
$D_{\widehat T^-} \cup \Big( \bigcup_{\ell \geq j} S_\ell(x,a)\Big)$,
where $S_\ell(x,a) = \tau_x^{-1}S_\ell(a)$. 
\item[1(b).] 
A  $(\delta,a)$-mark involving site $x \in D_{\widehat T^-}$. In this case, the
point $(x,\widehat T)$ is marked  $(\delta,a)$  and the set $D_{\widehat T}$ will be
$D_{\widehat T^-} \setminus \{x\}$. 
\item[2(a).] 
A  $(\circlearrowright,y,z)$-mark involving  $x \in D_{\widehat
    T^-}$ with $y<z$. In this  case, all the points of $D_{\widehat T}\cap
  [y,z] \times  \{T\}$ are marked with  $(\circlearrowright,y,z) $ and
  the  set $D_{\widehat T}$  will be  $ \sigma_{y,z}^{-1}(D_{\widehat  T^-}) =
  \sigma_{z,y}(D_{\widehat T^-})$.  
\item[2(b).] 
A  $(\circlearrowright,y,z)$-mark involving  $x \in D_{\widehat
    T^-}$ with $y>z$. In this  case, all the points of $D_{\widehat T}\cap
  [z,y] \times \{T\}$ are  marked with $ (\circlearrowright,y,z) $ and
  the  set $D_{\widehat T}$  will be  $ \sigma_{z,y}^{-1}(D_{\widehat  T^-}) =
  \sigma_{y,z}(D_{\widehat T^-})$.  
\end{itemize}
According to this construction, for each finite set $D$ and time $
t$,   we   are   defining   a   map   from   the   probability   space
$(\Omega,\FF,\PP)$  into the  space of  all possible  marked branching
structures on $\Z \times [\widehat 0, \widehat t]$,
\begin{equation}
\widehat  D_{[\widehat  0,\widehat  t]}^D  : (\omega,  t)  \mapsto  \left(N,
  (\widehat t_k, x_k,y_k,z_k,j_k,D_k), k=1,\dots,N \right),
\end{equation}
where 
\begin{itemize}
\item
$N$ is the number of marks in the interval $[\widehat 0, \widehat t]$,
\item
$\widehat t_k$ is the time of the occurrence of the $k$th mark,
\item
 $x_k$ is one site involved with the $k$th mark,
\item
  $y_k$ and  $z_k$ are the boundaries  of the sites  involved with the
  $k$th mark, if this mark is 
a circle arrow mark (if not $y_k=z_k=x_k$),
\item
 $j_k$ is the type of the $k$th mark   ($\delta$,   $(j,a)$, or 
$\circlearrowright$),
\item
$D_k$  is the set  of sites  in the
spatial projection  of the  structure between times  $ \widehat t_k  $ and
$\widehat t_{k+1}$.  
\end{itemize}
Note that  there are no  ambiguities if at  least two sites  $x,x' \in
D_{\widehat    T^-}$   are    involved    in   the    same   circle arrow    mark
$(\circlearrowright,y,z)$ with  $y<z$. Indeed in both  cases, the same
set of points  are marked with $(\circlearrowright,y,z) $  and the set
$D_{\widehat T}$ will be $ \sigma_{y,z}^{-1}(D_{\widehat T^-})$. It is similar
with a circle arrow   mark $(\circlearrowright,y,z)$ with $y>z$. 

One can check that if $D$ and  $D'$ are finite sets of sites such that
$D \subset D'$, then the marked branching structure 
\[
  \widehat  D_{[\widehat  0,\widehat  t]}^D (\omega, t) = \left(N,
  (\widehat t_k, x_k,y_k,z_k,j_k,D_k), k=1,\dots,N \right)
\]
is a subset of the marked branching structure 
\[
  \widehat  D_{[\widehat  0,\widehat  t]}^{D'} (\omega, t)=\left(N',
  (\widehat t'_k, x'_k,y'_k,z'_k,j'_k,D'_k), k=1,\dots,N' \right),
\]
in  the sense that  $N\leq N'$,  and for  any $k$  in $\{1,\dots,N\}$,
there exists $r(k)$ in $\{1,\dots,N'\}$ such that 
\[
  (\widehat      t_k,     x_k,y_k,z_k,j_k) = (\widehat t'_{r(k)},
  x'_{r(k)},y'_{r(k)},z'_{r(k)},j'_{r(k)}) 
  \quad \mbox{and} \quad 
  D_k \subset D'_{r(k)}.
\]
The central  idea of the construction is  exactly the same as in \cite{ferrari:ErgoSpinStirring}:
when  we go back  in time and
the generalized dual process meets a $(\delta,a)$-mark at site $x$, it is not necessary
to go further to know the value of $\eta(x)$, because it is determined
at that point by an independent random variable.
%
%
\subsection{Proofs of Theorems \ref{theo:Ergo} and \ref{theo:Ergo2}, 
and of Remarks \ref{rema:assu-1}\textit{(i)} and \ref{rema:assu-bis}\textit{(i)}} 
\label{sect:ProofThErgoDual}
\begin{proof} \textit{(of Theorem \ref{theo:Ergo}).}
The proof is  adapted from \cite{ferrari:ErgoSpinStirring}. If
the spatial projection of the dual 
structure started at time $t=\widehat 0$ is empty at time $0=\widehat t$, that
is,  $D_{\widehat t}^D= \emptyset$,  then $\eta_t(D)$  does not  depend on
$\eta_0=\eta$.  This  implies  that  a sufficient  condition  for  the
exponential ergodicity  of the  process is that,  for all  finite set $D$,
there exists positive constants $\alpha_1$, $\alpha_2$ such that 
\begin{equation} \label{equa:ErgoExpo}
  \PP(D_{\widehat t}^D= \emptyset) \leq \alpha_1 |D| \ee^{-t\alpha_2}.
\end{equation}
Hence, we are interested in  the evolution of the cardinal of  $D_{\widehat
  u}^D$. Since the  marks coming from the cut-and-paste process do
not change  the cardinal  of $D_{\widehat u}^D$  along time, they  are not
involved in the following.  

The process   $|D_{\widehat  u}^D|$ can  be  dominated  by a  
branching process $Z_u^{|D|} \in \N$,  that is, $ |D_{\widehat u}^D| \leq
Z_u^{|D|}  $ for  all  $u$  with probability  one.  This branching process is defined
as follows.    
At rate  $\blambda + \blambda_0$   
(defined in \eqref{equa:branch_sup-sup} and~\eqref{equa:bLambda0}) each  branch dies 
 and is  replaced by  either $s$  new  branches with
probability $\blambda/(\blambda + \blambda_0)$ or $0$ new
branches with probability $\blambda_0/(\blambda + \blambda_0)$.  

   Indeed, in   the    generalized   dual    process    presented   in
Section~\ref{sect:GeneDualProc}, a  site $x$ is  removed from $D_{\widehat
u}^D$ at time $\widehat T$  when a $(\delta,a)$-mark  appears at
$(x,\widehat T)$.  The  $(\delta,a)$-marks are distributed according to a PPP with rate $\lambda_0(a)$.
 Therefore, the total rate at which site $x$ is  removed from $D_{\widehat u}^D$ is  
 $\sum_{a \in  \AA}\lambda_0(a)=\blambda_0 $. 
Hence in the dominating branching process, a branch dies at rate $\blambda_0$. 

Now, we focus on the apparition of  new branches. In   the    generalized   dual    process,  
when a $(j,a)$-mark appears  at $(x,\widehat T)$, a maximum number of 
 $| \bigcup_{\ell \geq j} S_j(x,a)|=s_j(a)$ sites  might  be added  to  $D_{\widehat u}^D$.
 This happens at rate $\blambda_j (a)$. 
  Therefore for the dominating branching process, we use the bounds 
$\blambda_j (a)\leq \blambda$ and $s_j(a)\leq s$, so that
 a branch is replaced by $s$ branches at rate $\blambda$.  

The  initial state  of the  branching process  is  $Z_0^{|D|}=|D|$.  A
sufficient  condition for  \eqref{equa:ErgoExpo} is  that  the average
number of  branches created at each  branching is less  than $1$.  
This happens when 
 \begin{equation}\label{equa:condErgoExpo}
\frac{ s \blambda}
{  \blambda + \blambda_0} < 1.
\end{equation} 
\end{proof}
\begin{proof} \textit{(of Theorem \ref{theo:Ergo2}).}
The construction in this case of the generalized dual process follows the
line  of construction in Section \ref{sect:GeneDualProc},   except  that  one
should replace 1(a) by 
\begin{itemize}
\item[1(a).] 
A  $(j,a,i)$-mark  involving site  $x  \in  D_{\widehat  T^-}$ with  $j  \in
J^{(i)}(x,a)\setminus\{0\}$. In this case,  the point $(x,\widehat T)$ is marked
$(j,a,i)$ and the set $D_{\widehat T}$ will be $D_{\widehat T^-} \cup S^{(i)}(x,a)$,
where $S^{(i)}(x,a) = \tau_x^{-1}S^{(i)}(a)$. 
\end{itemize}
To prove \eqref{equa:ErgoFin} in Theorem~\ref{theo:Ergo}, we said that 
 a branch in  the dominating
branching process dies  and is either replaced by  $s$ or $0$ branches.
To prove here \eqref{equa:ErgoFin2} we say that
a branch in the dominating branching process dies at rate 
$\Theta=\overline \lambda_{0,d} + \sum_{i=1}^d \blambda^{(i)}$ 
(recall \eqref{equa:branch_inf_i} and \eqref{equa:branch_sup-sup_i}), 
and is either replaced by $s^{(1)}$, $\dots$, $s^{(d)}$ or $0$ branches
at  respective rates  $ \blambda^{(1)} /  \Theta$, $\dots$,  $
\blambda^{(d)} / \Theta$ and $ \blambda_{0,d}/\Theta$. 

A sufficient condition for \eqref{equa:ErgoExpo} is now
\begin{equation}\label{equa:condErgoExpo-i}
  \sum_{i=1}^d \frac{s^{(i)}  \blambda^{(i)}}{\Theta} < 1,
\end{equation} 
which is equivalent to \eqref{equa:ErgoFin2}. 
\end{proof}
 \begin{proof} \textit{(of Remarks \ref{rema:assu-1}(i) and \ref{rema:assu-bis}(i)).}
 To recover analogous results to Theorems 2.1 and 2.2 in \cite{ferrari:ErgoSpinStirring}
requires two main changes in the steps to prove Theorems \ref{theo:Ergo} and \ref{theo:Ergo2}:
We have to take a `uniform noise' when rewriting the generator $\LL_1$, 
then to take a different bound to define the birth rates of the branching process. 
We define the quantities we have to modify with an upper index ${\textsc F}$ 
(we keep the other ones unchanged), and consider in parallel the two results.

 We have to replace the definitions of $\blambda_0(a)$  in \eqref{equa:blambda}, 
and of $\blambda_{0}^{(i)}(a)$ in \eqref{equa:blambda_i}, by
\begin{equation}\label{def:blambda_0a-bis}
\blambda_0^{\textsc F}(a)=\lambda_0(a)-m,\qquad  \blambda_{0}^{(i),{\textsc F}}(a)=\lambda_{0}^{(i)}(a)-m^{(i)},
\end{equation}
the one of $\lambda (a)$ in \eqref{equa:Lambda}, and of $\lambda^{(i)} (a)$ 
in \eqref{equa:Lambda-i-total}, by
\begin{eqnarray}\label{equa:Lambda-bis}
  \lambda^{\textsc F} (a)&=&\max\big\{\lambda_j(a) \, : \, j \in J(a) \big\} -
  m = \lambda _{|J(a)| -1}- m, \\ \label{equa:Lambda-i-total-bis}
  \lambda^{(i),{\textsc F}} (a)&=&\max\big\{\lambda_j^{(i)}(a) \, : \, 
  j \in J^{(i)}(a) \big\} - m^{(i)} = \lambda^{(i)} _{|J^{(i)}(a)|-1}- m^{(i)};
  \end{eqnarray}
  also take, instead of \eqref{equa:bLambda0} and of \eqref{equa:branch_inf_i},
  \begin{equation}\label{equa:bLambda0-bis}
\blambda_0^{\textsc F}=|\AA| m,\qquad
    \blambda_{0,d}^{\textsc F}  = |\AA| \sum_{i=1}^d m^{(i)},
\end{equation}
and, instead of \eqref{equa:branch_sup-sup} and of \eqref{equa:branch_sup-sup_i},
\begin{eqnarray} \label{equa:branch_sup}
   \blambda^{\textsc F} &=& 
    \max_{a\in \AA} \Big(\blambda_0^{\textsc F}(a)
    + \sum_{j \in  J(a)\setminus\{0\}}   \blambda_j (a)\Big)
    =  \max_{a\in \AA}\lambda^{\textsc F} (a)
    =K-m,\\ \label{equa:branch_sup_i}
    \blambda^{(i),{\textsc F}} 
    &=&  \max_{a\in \AA} \Big(\blambda_{0}^{(i),{\textsc F}}(a)
    + \sum_{j \in  J^{(i)}(a)\setminus\{0\}} \blambda_j^{(i)} (a)\Big)
    =   \max_{a\in \AA}\lambda^{(i),{\textsc F}} (a)
    =K^{(i)}-m^{(i)}.
\end{eqnarray} 
Then, instead of \eqref{equa:Multinomial} and of \eqref{equa:Multinomial-i}, 
rewrite the pregenerator $\LL_1$ as
\begin{eqnarray} \nonumber 
\LL_1 f (\eta)  
& = & \sum_{x\in  \Z} \sum_{a \in \AA} \sum_{j \in J(a)\setminus\{0\}}
\blambda_j(a) \sum_{\ell 
    \geq j} \1_{\{ \eta \in A_\ell(x,a)\}}\big[f(\eta^x_a) - f(\eta)\big]  \\ \nonumber  
&&  +  \sum_{x\in  \Z} \sum_{a \in \AA} 
 \blambda_0^{\textsc F}(a) \sum_{\ell\geq 0} \1_{\{ \eta \in A_\ell(x,a)\}}\big[f(\eta^x_a) - f(\eta)\big]  \\    
    \label{equa:Multinomial-bis}
&&  +  \sum_{x \in  \Z}  \blambda_0^{\textsc F} \sum_{a \in  \AA}
\frac{1}{|\AA|} \left[  f(\eta^{x}_a) -
  f(\eta)\right].  
\end{eqnarray}
and
\begin{eqnarray} \nonumber 
  \LL_1  f (\eta)  
&=&  \sum_{i=1}^d  \sum_{x\in \Z}  \sum_{a  \in \AA}
  \sum_{j \in J^{(i)}(a)\setminus\{0\}} 
     \blambda_j^{(i)}(a)  \sum_{\ell  \geq   j}  \1_{\{\eta  \in
  A_\ell^{(i)}(x,a)\}}\big[f(\eta^x_a) -   f(\eta)\big] \\ \nonumber  
&&  + \sum_{i=1}^d  \sum_{x\in \Z}  \sum_{a  \in \AA}
     \blambda_{0}^{(i),{\textsc F}}(a)  \sum_{\ell  \geq 0}  \1_{\{\eta  \in
  A_\ell^{(i)}(x,a)\}}\big[f(\eta^x_a) -   f(\eta)\big]\\ \label{equa:Multinomial-bis-i}
&& +   
  \sum_{x \in  \Z}    \blambda_{0,d}^{\textsc F} \sum_{a \in  \AA}
  \frac{1}{|\AA|} \left[ f(\eta^{x}_a) - f(\eta) \right].
\end{eqnarray}
In the graphical construction, replace the definitions \eqref{eq:U0} and \eqref{eq:U0-i} by
\begin{equation}\label{eq:U0-bis}
\PP(U^{0,{\textsc F}}_n(x)=a) = \frac{1}{|\AA|};\qquad 
\PP(U^{0,d,{\textsc F}}_n(x)=a) = \frac{1}{|\AA|},
\end{equation}
so that these discrete random variables become uniform. 
Finally, the conditions for extinction \eqref{equa:condErgoExpo} 
and \eqref{equa:condErgoExpo-i} become
\begin{equation}\label{equa:condErgoExpo-bis}
\frac{ s \blambda^{\textsc F}}
{  \blambda^{\textsc F} + \blambda_0^{\textsc F}} < 1, \qquad
 \frac{\sum_{i=1}^d s^{(i)}  \blambda^{(i),{\textsc F}} }{\blambda_{0,d}^{\textsc F}+\sum_{i=1}^d \blambda^{(i),{\textsc F}}}<1.
\end{equation} 
\end{proof}
%
%
%
%
\section{Proofs through attractiveness  
} \label{sect:ProofsAttrErg}  
\begin{proof} \textit{(of Proposition \ref{prop:mu-a-t-c-g}).}
The lower and  upper invariant probability measures are 
$\mu_l=\lim_{t\to\infty}\delta_{\underline\cc}S(t)$ 
and $\mu_u=\lim_{t\to\infty}\delta_{\underline\gg}S(t)$, 
where $\delta_{\underline\cc}$ (resp. $\delta_{\underline\gg}$)
denotes the Dirac measure on the configuration $\eta$ 
such that $\eta(x)=\cc$ (resp. $\eta(x)=\gg$)
for all $x\in\Z$.  They are translation invariant. 
To prove ergodicity, we have to show that $\mu_l=\mu_u$. 
\\

\noindent $\bullet$ \textit{Step 1:} We derive consequences of attractiveness.\\
Note that the functions 
$\phi (\eta)=\1_{\{\eta(0)\geq\gg\}}$
and $\varphi (\eta)=\1_{\{\eta(0)>\cc\}}$ belong to $\MM$, 
and since $\cc$ and $\gg$ are respectively the smallest and largest elements of $\AA$
with respect to the order \eqref{def:order-on-AA_ctag}, we have 
$\phi (\eta)=\1_{\{\eta(0)=\gg\}}$
and $\varphi (\eta)=1-\1_{\{\eta(0=\cc\}}$,
hence
\begin{equation}\label{eq:c=g-ord}
\mu_l(\eta(0)=\gg)\leq\mu_u(\eta(0)=\gg) \qquad \mbox{and}\qquad 
\mu_u(\eta(0)=\cc)\leq\mu_l(\eta(0)=\cc).
\end{equation}
Similarly, the function 
$\psi (\eta)=\1_{\{\eta(0)>\tt\}}$
belongs to $\MM$, and, by the order \eqref{def:order-on-AA_ctag}, it satisfies
$\psi (\eta)=\1_{\{\eta(0)=\aa\}}+\1_{\{\eta(0)=\gg\}}=1-\1_{\{\eta(0)=\cc\}}-\1_{\{\eta(0)=\tt\}}$.
We thus have that 
\begin{eqnarray}\label{eq:psi-ag} 
\mu_l(\eta(0)=\aa)+\mu_l(\eta(0)=\gg)&\leq&\mu_u(\eta(0)=\aa)+\mu_u(\eta(0)=\gg),\\
\label{eq:psi-ct} 
\mu_l(\eta(0)=\cc)+\mu_l(\eta(0)=\tt)&\geq&\mu_u(\eta(0)=\cc)+\mu_u(\eta(0)=\tt). 
\end{eqnarray}
\noindent $\bullet$ \textit{Step 2:} We take into account the assumptions on the rates.\\
Combining \eqref{eq:c=g-ord}  with \eqref{eq:c=g} implies
\begin{equation}\label{eq:c=g-coupl}
\mu_l(\eta(0)=\cc)=\mu_u(\eta(0)=\cc)=\mu_l(\eta(0)=\gg)=\mu_u(\eta(0)=\gg).
\end{equation}
Then combining \eqref{eq:c=g-coupl}  first with \eqref{eq:psi-ag}, 
and then with \eqref{eq:psi-ct} yields
\begin{eqnarray}
\label{eq:psi-a} 
\mu_l(\eta(0)=\aa)&\leq&\mu_u(\eta(0)=\aa), \\  
\label{eq:psi-t} 
\mu_l(\eta(0)=\tt)&\geq&\mu_u(\eta(0)=\tt). 
\end{eqnarray}
Finally, combining  \eqref{eq:psi-a}--\eqref{eq:psi-t} with \eqref{eq:a=t} implies
\begin{equation}\label{eq:a=t-coupl}
\mu_l(\eta(0)=\tt)=\mu_u(\eta(0)=\tt)=\mu_l(\eta(0)=\aa)=\mu_u(\eta(0)=\aa).
\end{equation}
\noindent $\bullet$ \textit{Step 3:}
We now proceed in the same spirit as in Corollary II.2.8 in \cite{ligg:IPS}.
Let $\nu$ be a monotone coupling measure of $\mu_l$ and $\mu_u$.
It thus has to satisfy
\begin{equation}\label{reduction}
\nu((\eta,\xi):\eta(0)\neq \xi(0)) = \nu((\eta(0),\xi(0)) 
\in \{(\cc,\tt),(\cc,\aa),(\cc,\gg), (\tt, \aa), (\tt, \gg), (\aa,\gg) \}),
\end{equation}
so that by \eqref{eq:c=g-coupl},
\begin{eqnarray}\label{eq:nu-c}
\nu((\eta,\xi):\eta(0)=\xi(0)=\gg)&=&\mu_l(\eta(0)=\gg)=\mu_u(\xi(0)=\gg),\\ \label{eq:nu-g1}
\nu((\eta,\xi):\eta(0)=\xi(0)=\cc)&=&\mu_u(\xi(0)=\cc)=\mu_l(\eta(0)=\cc).
\end{eqnarray}
By  \eqref{reduction}, \eqref{eq:nu-c} and \eqref{eq:nu-g1}, and because 
$\cc$ and $\gg$ are respectively the smallest and largest elements of $\AA$
with respect to the order \eqref{def:order-on-AA_ctag},
\begin{eqnarray}\nonumber
\nu((\eta,\xi):\eta(0)\not=\gg,\xi(0)=\gg)
&=& \nu((\eta(0),\xi(0))\in\{(\cc,\gg),(\tt,\gg),(\aa,\gg)\})\\\label{eq:nu-diff-eg-g}
&=&\mu_u(\xi(0)=\gg)-\nu(\eta(0)=\xi(0)=\gg)=0,\\\nonumber
\nu((\eta,\xi):\eta(0)=\cc,\xi(0)\not=\cc)
&=& \nu((\eta(0),\xi(0))\in\{(\cc,\tt),(\cc,\aa),(\cc,\gg)\})\\\label{eq:nu-diff-eg-c}
&=&\mu_l(\eta(0)=\cc)-\nu(\eta(0)=\xi(0)=\cc)=0,
\end{eqnarray}
so that, using also again \eqref{reduction}
with respectively \eqref{eq:nu-diff-eg-c} and \eqref{eq:nu-diff-eg-g},
we obtain,
\begin{eqnarray}\nonumber
\nu((\eta,\xi):\eta(0)\not=\aa,\xi(0)=\aa)&=&
\nu((\eta(0),\xi(0))\in\{(\cc,\aa),(\tt,\aa)\})\\
\label{eq:nu-diff-eg-a1}&=&
0+\nu(\eta(0)=\tt,\xi(0)=\aa),\\
\nu((\eta,\xi):\eta(0)=\aa,\xi(0)\not=\aa)
&=&\nu((\eta(0),\xi(0))=(\aa,\gg))= 0. \label{eq:nu-diff-eg-a2}
\end{eqnarray}
Since we also have 
\begin{eqnarray}\label{eq:nu-diff-eg-a3}
\nu((\eta,\xi):\eta(0)\not=\aa,\xi(0)=\aa)
&=&\mu_u(\xi(0)=\aa)-\nu(\eta(0)=\xi(0)=\aa),\\\label{eq:nu-diff-eg-a4}
\nu((\eta,\xi):\eta(0)=\aa,\xi(0)\not=\aa)&=&\mu_l(\eta(0)=\aa)-\nu(\eta(0)=\xi(0)=\aa),
\end{eqnarray}
combining \eqref{eq:nu-diff-eg-a1}, \eqref{eq:nu-diff-eg-a2}  with \eqref{eq:a=t-coupl} implies that 
\begin{equation}\label{eq:nu-diff-eg-abis}
\nu(\eta(0)=\tt,\xi(0)=\aa)=0.
\end{equation}
Note that all the possible cases on the r.h.s. of \eqref{reduction} have probability $0$: 
$(\cc, \gg),(\tt, \gg),(\aa, \gg)$  by \eqref{eq:nu-diff-eg-g}, 
$(\cc, \tt), (\cc, \aa)$ by \eqref{eq:nu-diff-eg-c},  $(\tt, \aa)$ by \eqref{eq:nu-diff-eg-abis}. 
We conclude that 
\begin{equation}\nonumber
\nu((\eta,\xi):\eta(0)\neq \xi(0)) = 0,
\end{equation}
hence $\mu_l = \mu_u$.
\end{proof}

\noindent $\bullet$ \textbf{{Monotonicity of the RN+YpR model, 
and of the RN+YpR model with cut-and-paste mechanism}}
\begin{proof} \textit{(of Proposition \ref{prop:superp-attr}). }
We denote by  $\overline  {\LL_1}$  the  pregenerator of the monotone coupled dynamics
for the RN+YpR substitution process
(which exists by our assumption). 
Let similarly $\overline  \LL_2$ denote the  pregenerator of the coupled 
cut-and-paste dynamics through basic coupling,
that is, the same transition takes place for both copies of the process:
it is defined by, for a cylinder function $g$ on $X \times X$, 
\begin{equation} \label{equa:gene2-coupl}
\overline   \LL_2 g (\eta,\xi) = \sum_{x,y \in \Z} p(x,y) \big[
 g(\sigma_{x,y}(\eta),\sigma_{x,y}(\xi)) -g(\eta,\xi) \big].
\end{equation}
This coupled dynamics is monotone, since a transition does not change 
the way in which the values of the two processes on each site are coupled.
For the complete dynamics (that is the RN+YpR model with cut-and-paste mechanism) 
we consider the combination of both couplings, 
that is the pregenerator
\begin{equation} \label{equa:genesubs3-coupl}
\overline\LL g = \overline\LL_1 g+ \rho \overline\LL_2 g .
\end{equation}
Being the sum of two pregenerators of attractive dynamics, 
it yields also  an attractive dynamics.
We denote by $(\overline S(t),t\geq 0)$ its semi-group.
\end{proof}
\begin{proof} \textit{(of Proposition \ref{prop:RN+YpR-attr}).} 
To derive attractiveness, we construct a coupled dynamics $(\eta_t,\xi_t)_{t\geq 0}$ 
starting from ordered configurations $\eta_0\leq\xi_0$, 
through basic coupling. Then we find conditions on the rates
prohibiting the coupled transitions breaking
the increasing order between coupled configurations. We will denote by 
$\overline\LL_1$ the induced coupled generator (its existence was assumed 
in Proposition \ref{prop:superp-attr}). Similarly with \eqref{equa:genesubs1},
this generator  is defined on a cylinder 
function $f$ on $X\times X$ by 
\begin{equation} \label{equa:genesubs1-coupl}
  \overline\LL_1  f(\eta,\xi)  = \sum_{x  \in  \Z}  \sum_{(a,b)  \in \AA\times\AA}\  
  \overline c((a,b),\tau_x (\eta,\xi))
  \big[f(\eta^x_a,\xi^x_b) - f(\eta,\xi)\big].
\end{equation}
We now define the coupled rates $\overline c((a,b),\tau_x (\eta,\xi))$ for the 
transitions
$(\eta,\xi)\to(\eta^x_a,\xi^x_b)$.

By translation invariance of the dynamics, it is enough to look at site 0.
Thus we write the coupled transitions and their rates 
(according to basic coupling) in the following 3 tables.
There, we indicate with the symbol $(*)$ the coupled transitions 
to be forbidden for attractiveness;
we derive after each table the corresponding sufficient conditions
that these forbidden interactions induce. 

We rely on the rates given in Table \ref{tabl:RN+YpR}
for the RN+YpR model. 
%
\[
\begin{array}{lll|l}
  \multicolumn{3}{c|}{\mbox{Transition}}  &
  \multicolumn{1}{c}{\mbox{Rate}} \\
  \hline
  \hline
  (\cc,\cc) & \to & (\aa,\aa) & v_\aa \\
   & \to & (\gg,\gg) & v_\gg \\
  & \to & (\tt,\tt) & \min\{c(\tt,\eta),c(\tt,\xi)\} \\
  &     \to      &     (\cc,\tt)     &      c(\tt,\xi) -
  \min\{c(\tt,\eta),c(\tt,\xi)\} \\
  & \to & (\tt,\cc)  \quad (*) & c(\tt,\eta) -
  \min\{c(\tt,\eta),c(\tt,\xi)\} \\
  \hline
  (\cc,\tt) & \to & (\aa,\aa) & v_\aa\\
  & \to & (\gg,\gg) & v_\gg\\
  & \to & (\tt,\tt) & c(\tt,\eta)\\
  & \to & (\cc,\cc) & c(\cc,\xi)\\
  \hline
  (\cc,\aa) & \to & (\aa,\aa) & v_\aa\\
  & \to & (\tt,\tt) & v_\tt\\
  & \to & (\tt,\aa) & c(\tt,\eta) - v_\tt\\
  & \to & (\cc,\cc) & v_\cc\\
  & \to & (\gg,\gg) & v_\gg\\
  & \to & (\cc,\gg) & c(\gg,\xi) - v_\gg\\
  \hline
  (\cc,\gg) & \to & (\aa,\aa) & v_\aa\\
  & \to & (\cc,\aa) & c(\aa,\xi) - v_\aa\\
  & \to & (\tt,\tt) & v_\tt\\
  & \to & (\tt,\gg) & c(\tt,\eta) - v_\tt\\
  & \to & (\cc,\cc) & v_\cc\\
  & \to & (\gg,\gg) & v_\gg\\
  \end{array}
\]
%
Under basic coupling, both configurations
undergo the same transition according to the maximal possible rate, and then 
uncoupled transitions are added to fit the correct transitions for
each marginal. We detail this construction in the first 5 lines 
of this first table, 
the others are similar. There, we start from 
$(\eta(0),\xi(0))=(\cc,\cc)$. The rate for a transition from $\cc$ to 
$\aa$ or to  $\gg$ in Table \ref{tabl:RN+YpR} does not depend on the
value of the configuration on neighboring sites, therefore here
we have transitions respectively to $(\aa,\aa)$ or $(\gg,\gg)$ 
with the  rates $v_\aa$ or $v_\gg$, and these 
rates yield the correct rate for each marginal transition.
But the rate for a transition from $\cc$ to 
$\tt$ in Table \ref{tabl:RN+YpR} depends on the
value of the configuration on site $1$. Therefore the maximal rate for
a coupled transition from  $(\cc,\cc)$ to $(\tt,\tt)$ is 
$\min\{c(\tt,\eta),c(\tt,\xi)\}$, and, 
to obtain the correct rate for each marginal transition, it has to be supplemented 
by respective uncoupled transitions to $(\cc,\tt)$ and $(\tt,\cc)$, 
with  rates
$c(\tt,\xi) -  \min\{c(\tt,\eta),c(\tt,\xi)\}$ and 
$c(\tt,\eta) -  \min\{c(\tt,\eta),c(\tt,\xi)\}$. 

But a transition  to $(\tt,\cc)$ would break the
increasing order between the coupled configurations. To forbid it, 
that is, for its rate to be 0, we need to have 
\begin{equation}\label{eq:fbd-rate-1}
c(\tt,\eta)= \min\{c(\tt,\eta),c(\tt,\xi)\}.
\end{equation}
There are two possibilities, according to the value of the 
coupled configuration on $1$, which is such that $\eta(1)\leq\xi(1)$.
Either $\eta(1)\in Y=\{\cc,\tt\}$, hence $c(\tt,\eta)=w_\tt\leq c(\tt,\xi)$,
and \eqref{eq:fbd-rate-1} is satisfied;
or $\eta(1)\in R=\{\aa,\gg\}$,
hence $c(\tt,\eta)=w_\tt+r_\tt^{\eta(1)}$
and $c(\tt,\xi)=w_\tt+r_\tt^{\xi(1)}$: a necessary  and sufficient condition
for \eqref{eq:fbd-rate-1} to be satisfied is
\begin{equation*}
r_\tt^\aa\leq r_\tt^\gg.
\end{equation*}
%
\[
\begin{array}{lll|l}
  \multicolumn{3}{c|}{\mbox{Transition}}  &
  \multicolumn{1}{c}{\mbox{Rate}} \\
  \hline
  \hline
  (\tt,\tt) & \to & (\aa,\aa) & v_\aa \\
   & \to & (\cc,\cc) & \min\{c(\cc,\eta),c(\cc,\xi)\} \\
   & \to & (\cc,\tt) & c(\cc,\eta) - \min\{c(\cc,\eta),c(\cc,\xi)\}\\
   & \to & (\tt,\cc)  \quad (*) & c(\cc,\xi) - \min\{c(\cc,\eta),c(\cc,\xi)\} \\
   & \to & (\gg,\gg) & v_\gg \\
  \hline
  (\tt,\aa) & \to & (\aa,\aa) & v_\aa \\
   & \to & (\tt,\tt) & v_\tt \\
   & \to & (\cc,\cc) & v_\cc \\
   & \to & (\cc,\aa) & c(\cc,\eta) - v_\cc \\
   & \to & (\gg,\gg) & v_\gg \\
   & \to & (\tt,\gg) & c(\gg,\xi) - v_\gg \\
  \hline
  (\tt,\gg) & \to & (\aa,\aa) & v_\aa \\
  & \to & (\tt,\aa) & c(\aa,\xi) - v_\aa \\
  & \to & (\cc,\cc) & v_\cc \\
  & \to & (\cc,\gg) & c(\cc,\eta) - v_\cc \\
  & \to & (\tt,\tt) & v_\tt \\
  & \to & (\gg,\gg) & v_\gg \\
  \end{array}
\]
%
In this second table, the transition from $(\tt,\tt)$
to $(\tt,\cc)$ has to be forbidden for attractiveness, which requires 
\begin{equation}\label{eq:fbd-rate-2}
c(\cc,\xi) = \min\{c(\cc,\eta),c(\cc,\xi)\}.
\end{equation}
 We have that
 \begin{equation*}
c(\cc,\eta)= w_\cc+r_\cc^{\eta(1)}\1_{\{\eta(1)\in R\}}, \quad
c(\cc,\xi)=w_\cc+r_\cc^{\xi(1)}\1_{\{\xi(1)\in R\}},\quad\mbox{with}\quad
\eta(1)\leq\xi(1).
\end{equation*} 
Thus \eqref{eq:fbd-rate-2} is satisfied if  and only if
 \begin{equation*}
r_\cc^\aa = r_\cc^\gg = 0.
\end{equation*} 
%
\[
\begin{array}{lll|l}
  \multicolumn{3}{c|}{\mbox{Transition}}  &
  \multicolumn{1}{c}{\mbox{Rate}} \\
  \hline
  \hline 
  (\aa,\aa) & \to & (\tt,\tt) & v_\tt \\
  & \to & (\cc,\cc) & v_\cc \\
  & \to & (\gg,\gg) & \min\{c(\gg,\eta),c(\gg,\xi)\} \\
  & \to & (\aa,\gg) & c(\gg,\xi) - \min\{c(\gg,\eta),c(\gg,\xi)\} \\
  & \to & (\gg,\aa)  \quad (*)  & c(\gg,\eta) - \min\{c(\gg,\eta),c(\gg,\xi)\} \\
  \hline
  (\aa,\gg) & \to & (\aa,\aa) & c(\aa,\xi) \\
   & \to & (\gg,\gg) & c(\gg,\eta) \\
   & \to & (\tt,\tt) & v_\tt \\
   & \to & (\cc,\cc) & v_\cc \\ 
  \hline
  (\gg,\gg) & \to & (\aa,\aa) & \min\{c(\aa,\eta),c(\aa,\xi)\} \\
  & \to & (\aa,\gg) & c(\aa,\eta) - \min\{c(\aa,\eta),c(\aa,\xi)\} \\
  & \to & (\gg,\aa)  \quad (*)  & c(\aa,\xi) - \min\{c(\aa,\eta),c(\aa,\xi)\} \\
  & \to & (\tt,\tt) & v_\tt \\
  & \to & (\cc,\cc) & v_\cc \\
  \end{array}
\]
In this third table, the transitions from $(\aa,\aa)$ and from $(\gg,\gg)$ 
to $(\gg,\aa)$ have to be forbidden for attractiveness, which requires 
\begin{eqnarray}\label{eq:fbd-rate-3}
c(\gg,\eta) &=& \min\{c(\gg,\eta),c(\gg,\xi)\}, \\
\label{eq:fbd-rate-4}
c(\aa,\xi) &=& \min\{c(\aa,\eta),c(\aa,\xi)\}.
\end{eqnarray}
We have that
 \begin{eqnarray*}
c(\gg,\eta)&=&w_\gg+r_\gg^{\eta(-1)}\1_{\{\eta(-1)\in Y\}}, \,\,
c(\gg,\xi)= w_\gg+r_\gg^{\xi(-1)}\1_{\{\xi(-1)\in Y\}},\,\,\mbox{with}\,\,
\eta(-1)\leq\xi(-1),\\
c(\aa,\eta)&=& w_\aa+r_\aa^{\eta(-1)}\1_{\{\eta(-1)\in Y\}}, \,\,
c(\aa,\xi)= w_\aa+r_\aa^{\xi(-1)}\1_{\{\xi(-1)\in Y\}},\,\,\mbox{with}\,\,
\eta(-1)\leq\xi(-1).
\end{eqnarray*}
Thus \eqref{eq:fbd-rate-3} and \eqref{eq:fbd-rate-4}
are respectively
 satisfied if and only if
\begin{eqnarray*}
r_\gg^\cc = r_\gg^\tt&=&0,\\
r_\aa^\tt &\leq& r_\aa^\cc.
\end{eqnarray*} 
The proposition is proved.
\end{proof}
\begin{proof} \textit{(of Proposition \ref{prop:RN+YpR_mu-a-t-c-g}).}
We first compute
$\LL f (\eta)$ for $f (\eta)=\1_{\{\eta(0)=a\}}$, with $a\in\AA$.
 Recall that we write $c(b,\eta)$ for $c(0,b,\eta)$.
\begin{eqnarray}\nonumber
\LL_1 f (\eta)&=&\sum_{b\in\AA}c(b,\eta)[f (\eta_b^0)-f (\eta)]\\\label{eq:L1f-0}
&=&-\sum_{b\in\AA,\,b\not=a}c(b,\eta)f (\eta)+c(a,\eta)\1_{\{\eta(0)\not=a\}},\\\nonumber
\LL_2 f (\eta)&=& \sum_{x\in\Z} p(x,0) [\1_{\{\eta(x)=a\}}-\1_{\{\eta(0)=a\}}]
+\sum_{x,y\in\Z,\,x<0<y} p(x,y) [\1_{\{\eta(-1)=a\}}-\1_{\{\eta(0)=a\}}]\\\label{eq:L2f-0}
&&+\sum_{x,y\in\Z,\,y<0<x} p(x,y) [\1_{\{\eta(1)=a\}}-\1_{\{\eta(0)=a\}}].
\end{eqnarray}
Because $\mu$ is translation invariant, using \eqref{eq:L2f-0},
we have that $\int \LL_2 f (\eta)d\mu(\eta)=0$. Therefore,
by the invariance of $\mu$, we have $0=\int \LL f (\eta)d\mu(\eta)=\int \LL_1 f (\eta)d\mu(\eta)$.
Thus we only need to compute $\LL_1 f (\eta)$, 
relying on the values for the rates given in Table \ref{tabl:RN+YpR},
and using \eqref{eq:L1f-0}.
\begin{eqnarray}
\LL_1 (\1_{\{\eta(0)=\aa\}}) &=&
- \1_{\{\eta(0)=\aa\}}(v_\tt + v_\cc+w_\gg) 
- \sum_{a\in Y}\1_{\{(\eta(-1),\eta(0))=(a,\aa)\}}r_\gg^a \nonumber\\
&& + \1_{\{\eta(0)\in Y\}}v_\aa + \1_{\{\eta(0)=\gg\}}w_\aa
+ \sum_{a\in Y} \1_{\{(\eta(-1),\eta(0))=(a,\gg)\}}r_\aa^a, \label{eq:RN+YpR_L1f-0-aa}\\
\LL_1 (\1_{\{\eta(0)=\gg\}}) &=&
- \1_{\{\eta(0)=\gg\}}(w_\aa + v_\tt+v_\cc)
- \sum_{a\in Y} \1_{\{(\eta(-1),\eta(0))=(a,\gg)\}}r_\aa^a\nonumber\\
&& +  \1_{\{\eta(0)\in Y\}}v_\gg +  \1_{\{\eta(0)=\aa\}}w_\gg
+ \sum_{a\in Y} \1_{\{(\eta(-1),\eta(0))=(a,\aa)\}}r_\gg^a, \label{eq:RN+YpR_L1f-0-gg}\\
\LL_1 (\1_{\{\eta(0)=\cc\}}) &=&
- \1_{\{\eta(0)=\cc\}} (v_\aa + v_\gg+w_\tt) 
- \sum_{a\in R} \1_{\{(\eta(0),\eta(1))=(\cc,a)\}}r_\tt^a \nonumber\\
&& +  \1_{\{\eta(0)\in R\}} v_\cc +  \1_{\{\eta(0)=\tt\}}w_\cc
 + \sum_{a\in R} \1_{\{(\eta(0),\eta(1))=(\tt,a)\}} r_\cc^a, \label{eq:RN+YpR_L1f-0-cc}\\
\LL_1 (\1_{\{\eta(0)=\tt\}}) &=&
- \1_{\{\eta(0)=\tt\}}(v_\aa + w_\cc + v_\gg)
- \sum_{a\in R} \1_{\{(\eta(0),\eta(1))=(\tt,a)\}} r_\cc^a \nonumber\\
&&+ \1_{\{\eta(0)\in R\}} v_\tt + \1_{\{\eta(0)=\cc\}} w_\tt
+  \sum_{a\in R} \1_{\{(\eta(0),\eta(1))=(\cc,a)\}} r_\tt^a. \label{eq:RN+YpR_L1f-0-tt}
\end{eqnarray}
We write $0=\int \LL_1 f (\eta)d\mu(\eta)$ 
starting from \eqref{eq:RN+YpR_L1f-0-aa}--\eqref{eq:RN+YpR_L1f-0-tt}. This
gives the following linear system, whose last line states that $\mu$ is a probability measure.
\begin{equation}\label{eq:L1f-0-aa-mu}
\begin{cases}
-(v_{\textsc y}+w_\gg) \mu(\eta(0)=\aa)+ w_\aa \mu(\eta(0)=\gg) + v_\aa \mu(Y) &= r_{\textsc r},\\
w_\gg \mu(\eta(0)=\aa)  - (v_{\textsc y}+w_\aa) \mu(\eta(0)=\gg) + v_\gg \mu(Y) &= -r_{\textsc r}, \\
- (v_{\textsc r}+w_\tt) \mu(\eta(0)=\cc) + w_\cc \mu(\eta(0)=\tt)+v_\cc \mu(R) &= r_{\textsc y}, \\
w_\tt \mu(\eta(0)=\cc) - (v_{\textsc r}+w_\cc)  \mu(\eta(0)=\tt)+v_\tt \mu(R) &= -r_{\textsc y}, \\
\mu(\eta(0)=\aa) + \mu(\eta(0)=\gg) &= \mu(R), \\
\mu(\eta(0)=\cc) + \mu(\eta(0)=\tt) &= \mu(Y),  \\
\mu(\eta(0)=\aa) + \mu(\eta(0)=\gg) + \mu(\eta(0)=\cc) + \mu(\eta(0)=\tt) &= 1.
\end{cases}
\end{equation}
Combining the addition of lines 1 and 2 (taking into account lines 5 and 6)
with the last line of \eqref{eq:L1f-0-aa-mu} gives a system whose solution is 
\eqref{eq:RN+YpR_mu-a+g}--\eqref{eq:RN+YpR_mu-c+t}. We then insert those values
into \eqref{eq:L1f-0-aa-mu}. Solving the system composed by its lines 1 
and 5 (resp. its lines 3 and 6) yields \eqref{eq:RN+YpR_mu-a}, \eqref{eq:RN+YpR_mu-g}
(resp. \eqref{eq:RN+YpR_mu-c}, \eqref{eq:RN+YpR_mu-t}). 
\end{proof} 
\begin{proof} \textit{(of Proposition \ref{prop:mu-a-t_or_c-g}).}
We assume that \eqref{eq:c=g} is satisfied; the case where \eqref{eq:a=t} 
is satisfied is similar, and its proof is left to the reader.\\
We go through the 3 steps of the proof of Proposition \ref{prop:mu-a-t-c-g}: Step 1 is still valid, with
\eqref{eq:psi-ag}--\eqref{eq:psi-ct} which become 
\begin{eqnarray}\label{eq:psi-ag-more} 
\mu_l(\eta(0)=\aa)+\mu_l(\eta(0)=\gg)&=&\mu_u(\eta(0)=\aa)+\mu_u(\eta(0)=\gg),\\
\label{eq:psi-ct-more} 
\mu_l(\eta(0)=\cc)+\mu_l(\eta(0)=\tt)&=&\mu_u(\eta(0)=\cc)+\mu_u(\eta(0)=\tt).
\end{eqnarray}
because of \eqref{eq:RN+YpR_mu-a+g}--\eqref{eq:RN+YpR_mu-c+t}. Thus, in Step 2,
 \eqref{eq:c=g-coupl} is still valid, as well as
 \eqref{eq:psi-a}--\eqref{eq:psi-t}. But  combining them with 
 \eqref{eq:psi-ag-more}--\eqref{eq:psi-ct-more} implies that 
\begin{equation}\label{eq:a-and-t}
\mu_l(\eta(0)=\tt)=\mu_u(\eta(0)=\tt)\quad\hbox{and}\quad\mu_l(\eta(0)=\aa)=\mu_u(\eta(0)=\aa).
\end{equation}
Finally, Step 3 
 is valid, which ends the proof.
\end{proof}
\begin{proof} \textit{(of Proposition \ref{prop:nu-some-coupl-0}).}
Because $\nu$ is a monotone coupling measure  of  $\mu_1$ and $\mu_2$, it satisfies
\eqref{eq:nu-strassen}  and \eqref{reduction}.
Thus  we have that
\begin{equation}\label{eq:mu_l-u_c}
\begin{cases}
\mu_1(\eta(0)=\cc) &= \nu ( (\eta(0),\xi(0)) \in \{ (\cc,\cc),(\cc,\tt),(\cc,\aa),(\cc,\gg) \} ),\\
\mu_2(\xi(0)=\cc)  &= \nu ( (\eta(0),\xi(0)) = (\cc,\cc) ), 
\end{cases}
\end{equation}
\begin{equation}\label{eq:mu_l-u_t}
\begin{cases}
\mu_1(\eta(0)=\tt) &= \nu ( (\eta(0),\xi(0)) \in \{ (\tt,\tt),(\tt,\aa),(\tt,\gg)\} ),\\
\mu_2(\xi(0)=\tt)  &= \nu ( (\eta(0),\xi(0)) \in \{ (\cc,\tt),(\tt,\tt) \} ), 
\end{cases}
\end{equation}
\begin{equation}\label{eq:mu_l-u_a}
\begin{cases}
\mu_1(\eta(0)=\aa) &= \nu ( (\eta(0),\xi(0)) \in \{ (\aa,\aa),(\aa,\gg) \} ),\\
\mu_2(\xi(0)=\aa)  &= \nu ( (\eta(0),\xi(0)) \in \{ (\cc,\aa),(\tt,\aa),(\aa,\aa) \} ), 
\end{cases}
\end{equation}
\begin{equation}\label{eq:mu_l-u_g}
\begin{cases}
\mu_1(\eta(0)=\gg) &= \nu ( (\eta(0),\xi(0)) = (\gg,\gg) ),\\
\mu_2(\xi(0)=\gg)  &= \nu ( (\eta(0),\xi(0))  \in \{(\cc,\gg),(\tt,\gg),(\aa,\gg),(\gg,\gg)\} ). 
\end{cases}
\end{equation}
Hence 
\begin{eqnarray*}
\mu_1(\eta(0)\in\{\aa,\gg\} ) &=& \nu ( (\eta(0),\xi(0)) \in \{ (\aa,\aa),(\aa,\gg), (\gg,\gg) \} ),\\
\mu_2(\xi(0)\in\{\aa,\gg\} )  &=& \nu ( (\eta(0),\xi(0)) \in \{ (\cc,\aa),(\tt,\aa),(\aa,\aa),(\cc,\gg),(\tt,\gg),(\aa,\gg),(\gg,\gg) \}. )
\end{eqnarray*}
We then conclude thanks to \eqref{eq:RN+YpR_mu-a+g}--\eqref{eq:RN+YpR_mu-c+t}
that yield 
\[\mu_1(\eta(0)\in\{\aa,\gg\} ) = \mu_2(\xi(0)\in\{\aa,\gg\} )  = \frac{v_{\textsc r}}{v_{\textsc y}+v_{\textsc r}}\, .\]
Finally, \eqref{eq:mu_l-u_c}--\eqref{eq:mu_l-u_g} can be simplified in
\begin{equation}\label{eq:mu_l-u_a-c-g-t}
\begin{cases}
\mu_1(\eta(0)=\cc) &= \nu ( (\eta(0),\xi(0)) \in \{ (\cc,\cc),(\cc,\tt) \} ),\\
\mu_2(\xi(0)=\cc)  &= \nu ( (\eta(0),\xi(0)) = (\cc,\cc) ),\\
\mu_1(\eta(0)=\tt) &= \nu ( (\eta(0),\xi(0)) = (\tt,\tt) ),\\
\mu_2(\xi(0)=\tt)  &= \nu ( (\eta(0),\xi(0)) \in \{ (\cc,\tt),(\tt,\tt) \} ),\\
\mu_1(\eta(0)=\gg) &= \nu ( (\eta(0),\xi(0)) = (\gg,\gg) ),\\
\mu_2(\xi(0)=\gg)  &= \nu ( (\eta(0),\xi(0)) \in \{ (\aa,\gg),(\gg,\gg) \} ),\\
\mu_1(\eta(0)=\aa) &= \nu ( (\eta(0),\xi(0)) \in \{ (\aa,\aa),(\aa,\gg) \} ),\\
\mu_2(\xi(0)=\aa)  &= \nu ( (\eta(0),\xi(0)) = (\aa,\aa) ) .
\end{cases}
\end{equation}
\end{proof}
\begin{proof} \textit{(of Proposition \ref{prop:nu-diag}).}

\noindent$\bullet$
We compute $\overline\LL_1 f(\eta,\xi)$ for the function 
$f(\eta,\xi)=\1_{\{\eta(0)=a,\xi(0)=b\}}$, for $a,b\in\AA$.
 We will then apply this computation to $(a,b)=(\aa,\gg)$ 
 and to $(a,b)=(\cc,\tt)$. 
We have (recall the notation \eqref{equa:genesubs1-coupl})
\begin{eqnarray}\nonumber
 \overline\LL_1  f(\eta,\xi)  &=&  \sum_{(a',b')  \in \AA\times\AA}\  
  \overline c((a',b'), (\eta,\xi))
  \big[f(\eta^0_{a'},\xi^0_{b'}) - f(\eta,\xi)\big]\\\nonumber
&=&  - \sum_{(a',b')  \not= (a,b)}\  
  \overline c((a',b'), (\eta,\xi)) f(\eta,\xi)
  +
  \overline c((a,b), (\eta,\xi))
  \big[1 - f(\eta,\xi)\big]\\\label{eq:barL1-ab}
&=&  - \sum_{(a',b')  \in \AA\times\AA}\  
  \overline c((a',b'), (\eta,\xi)) f(\eta,\xi)
  +
  \overline c((a,b), (\eta,\xi))
\end{eqnarray}
where we used 
\[
f(\eta^0_{a'},\xi^0_{b'})=\1_{\{\eta^0_{a'}(0)=a,\xi^0_{b'}(0)=b\}}
=\1_{\{a'=a,b'=b\}}.
\]
\noindent$\bullet$
We now apply \eqref{eq:barL1-ab} to $(a,b)=(\aa,\gg)$,  
according to the coupling rates 
given in the tables in the proof of Proposition \ref{prop:RN+YpR-attr}, 
then using the attractiveness conditions
from Proposition \ref{prop:RN+YpR-attr}. 
\begin{eqnarray}\nonumber
 \overline\LL_1  f(\eta,\xi)  &=&  - \1_{\{\eta(0)=\aa,\xi(0)=\gg\}}  \left[c(\aa,\xi)+c(\gg,\eta)+v_\tt+v_\cc\right]
\\\nonumber
&&  + \1_{\{\eta(0)=\aa,\xi(0)=\aa\}}  \left[c(\gg,\xi)-\min\left\{c(\gg,\eta),c(\gg,\xi) \right\}\right]
\\\nonumber
&&  +\1_{\{\eta(0)=\gg,\xi(0)=\gg\}}  \left[c(\aa,\eta)-\min\left\{c(\aa,\eta),c(\aa,\xi) \right\}\right]
\\\nonumber
&=&  - \1_{\{\eta(0)=\aa,\xi(0)=\gg\}}  \left[w_\aa+\1_{\{\xi(-1)=a\in Y\}}r_\aa^a+w_\gg+0+v_\tt+v_\cc\right]  + 0 
\\\nonumber
&& +\1_{\{\eta(0)=\gg,\xi(0)=\gg\}}  \left[\1_{\{\eta(-1)=a\in Y\}}r_\aa^a
-\min\left\{\1_{\{\eta(-1)=a\in Y\}}r_\aa^a,\1_{\{\xi(-1)=b\in Y\}}r_\aa^b \right\}\right]
\\\nonumber
&=&  - \1_{\{\eta(0)=\aa,\xi(0)=\gg\}}  \left[v_\tt+v_\cc+w_\gg+w_\aa+\1_{\{\xi(-1)=\cc\}}r_\aa^\cc+\1_{\{\xi(-1)=\tt\}}r_\aa^\tt\right]  
\\\label{eq:barL1-aa-gg}
&& +\1_{\{\eta(-1)=\cc,\xi(-1)=\tt,\eta(0)=\gg,\xi(0)=\gg\}}  \left[r_\aa^\cc
-r_\aa^\tt \right].
\end{eqnarray}
For the last term in the last equality, we have used that 
since $\eta(-1)\leq\xi(-1)$, either they are both equal, or
$\eta(-1)=\cc,\xi(-1)=\tt$. \\
We now integrate \eqref{eq:barL1-aa-gg} with respect to $\nu$, using also
  translation invariance of $\nu$. 
To lighten the formulas, we use the  notation, for $x,y\in\Z$, $a_1, a_2,b_1, b_2\in\AA$,
\[ \nu\left(\begin{matrix}x\, &\, y\\
                             a_1\, &\, a_2\\
                             b_1\, &\, b_2\end{matrix}\right) 
=\nu( \eta(x)=a_1, \eta(y)=a_2,\xi(x)=b_1, \xi(y)=b_2) \]
as well as
\[ \nu\left(\begin{matrix}x\\
                             a_1\\
                             b_1\end{matrix}\right) 
=\nu( \eta(x)=a_1,\xi(x)=b_1), \quad 
\nu\left(\begin{matrix}x\, &\, y\\
                             a_1\, &\, *\\
                             b_1\, &\, b_2\end{matrix}\right) 
=\nu( \eta(x)=a_1, \xi(x)=b_1, \xi(y)=b_2). \]
We  also  use that by Proposition \ref{prop:nu-some-coupl-0} we have 
\begin{eqnarray*}
\nu \left(\begin{matrix}0\, &\, 1\\
                             \cc\, &\, \gg\\
                             \tt\, &\, \gg\end{matrix}\right)
=    \nu \left(\begin{matrix}0\, &\, 1\\
                             \cc\, &\, \gg\\
                             \tt\, &\, *\end{matrix}\right), \quad
   \nu \left(\begin{matrix}0\, &\, 1\\
                             \cc\, &\, \aa\\
                             \cc\, &\, \gg\end{matrix}\right)
=    \nu \left(\begin{matrix}0\, &\, 1\\
                             *\, &\, \aa\\
                             \cc\, &\, \gg\end{matrix}\right), \quad
     \nu \left(\begin{matrix}0\, &\, 1\\
                             \tt\, &\, \aa\\
                             \tt\, &\, \gg\end{matrix}\right)
=    \nu \left(\begin{matrix}0\, &\, 1\\
                             \tt\, &\, \aa\\
                             *\, &\, \gg\end{matrix}\right).                          
\end{eqnarray*}
We thus get 
\begin{eqnarray}
\nonumber
 0=  \int \overline\LL_1  f(\eta,\xi) d\nu 
 &=& -[v_\tt + v_\cc+w_\gg + w_\aa]\, \nu \left(\begin{matrix}0\\
                             \aa\\
                             \gg\end{matrix}\right)
                             - r_\aa^\tt \,\nu \left(\begin{matrix}0\, &\, 1\\
                             \cc\, &\, \aa\\
                             \tt\, &\, \gg\end{matrix}\right) 
 - r_\aa^\cc\, \nu \left(\begin{matrix}0\, &\, 1\\
                             *\, &\, \aa\\
                             \cc\, &\, \gg\end{matrix}\right)
  \\ \label{eq:barL1-a-g}
  && 
   - r_\aa^\tt \,\nu \left(\begin{matrix}0\, &\, 1\\
                             \tt\, &\, \aa\\
                             *\, &\, \gg\end{matrix}\right)
  + [ r_\aa^\cc - r_\aa^\tt ] \,\nu \left(\begin{matrix}0\, &\, 1\\
                             \cc\, &\, \gg\\
                             \tt\, &\, *\end{matrix}\right).                                                     
\end{eqnarray}
\noindent$\bullet$
We   similarly  apply \eqref{eq:barL1-ab} to $(a,b)=(\cc,\tt)$, 
  then also integrate  with respect to $\nu$,
using translation invariance of $\nu$.  This gives 
\begin{eqnarray}
\nonumber
 0= \int \overline\LL_1  f(\eta,\xi) d\nu 
 &=& -[w_\tt + w_\cc+v_\gg + v_\aa] \,\nu \left(\begin{matrix}0\\
                             \cc\\
                             \tt\end{matrix}\right) 
 - r^\aa_\tt  \,\nu \left(\begin{matrix}0\, &\, 1\\
                                        \cc\, &\, \aa\\
                                        \tt\, &\, \gg\end{matrix}\right)  
 - r^\aa_\tt \, \nu \left(\begin{matrix}0\, &\, 1\\
                                        \cc\, &\, *\\
                                        \tt\, &\, \aa\end{matrix}\right) \\ \label{eq:barL1-c-t}
  &&  - r^\gg_\tt  \,\nu \left(\begin{matrix}0\, &\, 1\\
                                            \cc\, &\, \gg\\
                                            \tt\, &\, *\end{matrix}\right) 
      + [ r_\tt^\gg - r_\tt^\aa]\, \nu \left(\begin{matrix}0\, &\, 1\\
                                                           *\, &\, \aa\\
                                                          \cc\, &\, \gg\end{matrix}\right).    
\end{eqnarray}
\noindent$\bullet$
Adding \eqref{eq:barL1-a-g} and \eqref{eq:barL1-c-t} gives:
\begin{eqnarray}\nonumber
0 
&=& -[v_\tt + v_\cc+w_\gg + w_\aa] \,\nu \left(\begin{matrix}0\\
                                               \aa\\
                                               \gg\end{matrix}\right)
    + [ (r_\tt^\gg - r_\tt^\aa)- r_\aa^\cc] \,\nu \left(\begin{matrix}0\, &\, 1\\
                                                                      *\, &\, \aa\\
                                                                      \cc\, &\, \gg\end{matrix}\right) 
                              \\ \nonumber
 &&  -[w_\tt + w_\cc+v_\gg + v_\aa]\,\nu\left(\begin{matrix}0\\
                                                            \cc\\
                                                            \tt\end{matrix}\right)
  + [ (r_\aa^\cc - r_\aa^\tt ) - r^\gg_\tt ] \,\nu \left(\begin{matrix}0\, &\, 1\\
                             \cc\, &\, \gg\\
                             \tt\, &\, *\end{matrix}\right)  \\ \label{eq:ag+ct} 
&&  - [r_\aa^\tt+r^\aa_\tt ] \,\nu \left(\begin{matrix}0\, &\, 1\\
                             \cc\, &\, \aa\\
                             \tt\, &\, \gg\end{matrix}\right)   
     - r_\aa^\tt \,\nu \left(\begin{matrix}0\, &\, 1\\
                             \tt\, &\, \aa\\
                             *\, &\, \gg\end{matrix}\right)
   - r^\aa_\tt \,\nu \left(\begin{matrix}0\, &\, 1\\
                             \cc\, &\, *\\
                             \tt\, &\, \aa\end{matrix}\right).                       
\end{eqnarray}
 \noindent$\bullet$
We claim that 
\begin{equation}\label{eq:ag0equivct0}
\nu ( (\eta(0),\xi(0))=(\aa,\gg) )=0\quad\mbox{if and only if}\quad
\nu ( (\eta(0),\xi(0))=(\cc,\tt) )=0.
\end{equation} 
Indeed, assuming that $\nu ( (\eta(0),\xi(0))=(\aa,\gg) )=0$, we have also
\begin{equation*}\label{eq:also0forag}
\nu\left(\begin{matrix}0\, &\, 1\\
                             \cc\, &\, \aa\\
                             \tt\, &\, \gg\end{matrix}\right) 
=\nu\left(\begin{matrix}0\, &\, 1\\
                             \tt\, &\, \aa\\
                             *\, &\, \gg\end{matrix}\right)
  =\nu\left(\begin{matrix}0\, &\, 1\\
                             *\, &\, \aa\\
                             \cc\, &\, \gg\end{matrix}\right)
                             =0.
\end{equation*} 
Hence \eqref{eq:barL1-c-t} becomes  
\begin{eqnarray*}
\nonumber
 0 &=& -[w_\tt + w_\cc+v_\gg + v_\aa] \,\nu \left(\begin{matrix}0\\
                             \cc\\
                             \tt\end{matrix}\right)  
 - r^\aa_\tt \, \nu \left(\begin{matrix}0\, &\, 1\\
                                        \cc\, &\, *\\
                                        \tt\, &\, \aa\end{matrix}\right) 
                                          - r^\gg_\tt  \,\nu \left(\begin{matrix}0\, &\, 1\\
                                            \cc\, &\, \gg\\
                                            \tt\, &\, *\end{matrix}\right).     
\end{eqnarray*}
The above r.h.s. contains only  non-positive  terms: 
it means that each of them is equal to 0, 
in particular the first one, for which 
we know that $w_\tt + w_\cc+v_\gg + v_\aa>0$ 
 (note that for the second and third terms,
$r^\aa_\tt$ and/or $r^\gg_\tt$ could be equal to 0).  
This implies that $\nu ( (\eta(0),\xi(0))=(\cc,\tt) )=0$. \\
Similarly, assuming that $\nu ( (\eta(0),\xi(0))=(\cc,\tt) )=0$, 
\eqref{eq:barL1-a-g} becomes
\begin{eqnarray*}
 0=  -[v_\tt + v_\cc+w_\gg + w_\aa]\, \nu \left(\begin{matrix}0\\
                             \aa\\
                             \gg\end{matrix}\right)
                             - r_\aa^\tt \,\nu \left(\begin{matrix}0\, &\, 1\\
                             \cc\, &\, \aa\\
                             \tt\, &\, \gg\end{matrix}\right) 
   - r_\aa^\tt \,\nu \left(\begin{matrix}0\, &\, 1\\
                             \tt\, &\, \aa\\
                             *\, &\, \gg\end{matrix}\right).                                              
\end{eqnarray*}
Since $v_\tt + v_\cc+w_\gg + w_\aa>0$, this implies 
that $\nu ( (\eta(0),\xi(0))=(\aa,\gg) )=0$. \\ 

\noindent$\bullet$
We then consider each of the 3 assumptions on the rates.

(a) Assuming that $r_\aa^\cc = r_\aa^\tt$,  
we have that the r.h.s. in \eqref{eq:barL1-a-g}
contains only  non-positive  terms: 
it means that each of them is equal to 0, in particular the first one, 
for which  $v_\tt + v_\cc+w_\gg + w_\aa>0$. This implies that 
$\nu ( (\eta(0),\xi(0))=(\aa,\gg) )=0$, 
 and we conclude thanks to \eqref{eq:ag0equivct0}.  

(b) Assuming that  $r_\tt^\gg = r_\tt^\aa$  induces a similar reasoning, 
to get first $\nu ( (\eta(0),\xi(0))=(\cc,\tt) )=0$
by considering \eqref{eq:barL1-c-t} and using that
$w_\tt + w_\cc+v_\gg + v_\aa>0$, then 
 conclude thanks to \eqref{eq:ag0equivct0}.  

 (c), \textit{(i)}  
As a preliminary,  we examine the conditions  $(\alpha)$, $(\beta)$, $(\gamma)$,
 $(\delta)$: It is impossible to have neither $(\alpha)$ nor $(\gamma)$ satisfied, since it
 would imply 
 \[
 r_\tt^\gg - r_\tt^\aa > r_\aa^\cc >  r^\gg_\tt + r_\aa^\tt \quad\mbox{ hence }\quad
  - r_\tt^\aa > r_\aa^\tt .
  \]
Thus if one of them is not satisfied, the other automatically is. \\
If $(\beta)$ is satisfied then $(\alpha)$ is not satisfied, 
hence $(\gamma)$ is satisfied. \\
Similarly, if $(\delta)$ is satisfied then $(\gamma)$ is not satisfied, 
hence $(\alpha)$ is satisfied. 
 
\textit{(ii)}
 If $(\alpha)$ and $(\gamma)$  are satisfied, then the r.h.s. of \eqref{eq:ag+ct} 
contains only non-positive terms, hence each of them is equal to 0. Since 
$v_\tt + v_\cc+w_\gg + w_\aa>0$ and $w_\tt + w_\cc+v_\gg + v_\aa>0$, this implies that 
$\nu ( (\eta(0),\xi(0))=(\aa,\gg) )=\nu ( (\eta(0),\xi(0))=(\cc,\tt) )=0$. 

\textit{(iii)} If $(\beta)$ is satisfied, on the one hand
we  bound the sum of the first two terms in the r.h.s. of \eqref{eq:ag+ct}  by
\begin{equation*}                                            
\left(-[v_\tt + v_\cc+w_\gg + w_\aa] 
    + [ (r_\tt^\gg - r_\tt^\aa)- r_\aa^\cc] \right) \,\nu \left(\begin{matrix}0\\
                                               \aa\\
                                               \gg\end{matrix}\right) ,           
\end{equation*}
and on the other hand, by \textit{(i)}, $(\gamma)$ is satisfied. Hence 
the r.h.s. of \eqref{eq:ag+ct}, which is equal to 0, is bounded by a sum of only
non-positive terms, thus each of them is equal to 0. Since $w_\tt + w_\cc+v_\gg + v_\aa>0$,
this implies that $\nu ( (\eta(0),\xi(0))=(\cc,\tt) )=0$. We
 conclude thanks to \eqref{eq:ag0equivct0}. 
 
 \textit{(iv)} Similarly, if $(\delta)$ is satisfied, on the one hand
 we  bound the sum of the third and fourth terms in the r.h.s. of \eqref{eq:ag+ct} by
\begin{equation*}
\left(-[w_\tt + w_\cc+v_\gg + v_\aa]
  + [ (r_\aa^\cc - r_\aa^\tt ) - r^\gg_\tt ] \right) \,\,\nu\left(\begin{matrix}0\\
                                                            \cc\\
                                                            \tt\end{matrix}\right) , 
\end{equation*}  
and on the other hand, by \textit{(i)}, $(\alpha)$ is satisfied. Hence 
the r.h.s. of \eqref{eq:ag+ct}, which is equal to 0, is bounded by a sum of only
non-positive terms, thus each of them is equal to 0. Since 
$v_\tt + v_\cc+w_\gg + w_\aa>0$, this implies that 
$\nu ( (\eta(0),\xi(0))=(\aa,\gg) )=0$. We
 conclude thanks to \eqref{eq:ag0equivct0}. 
 
 The proposition is proved.
 \end{proof} 
\noindent{\bf Acknowledgements.}  
We thank the referee for carefully reading the first version of this text,
and for giving us several helpful comments. We thank both
MAP5 lab. and TU M\"unchen for financial support and hospitality. 
\bibliographystyle{alpha}
\bibliography{Biblio-enm}
\end{document}